\documentclass{amsart}
\usepackage{color}

\newtheorem{theorem}{Theorem}[section]
\newtheorem{lemma}[theorem]{Lemma}

\theoremstyle{definition}
\newtheorem{definition}[theorem]{Definition}

\theoremstyle{remark}
\newtheorem{remark}[theorem]{Remark}

\numberwithin{equation}{section}

\begin{document}

\title{The Stability of Full Dimensional KAM tori for Nonlinear Schr\"odinger equation}
\author{Hongzi Cong}
\address{School of Mathematical Sciences, Dalian University of Technology, Dalian, Liaoning 116024, China}
\email{conghongzi@dlut.edu.cn}%    Information for first author
\author{Jianjun Liu}
\address{School of Mathematical Sciences, Sichuan University, Chengdu, Sichuan 610065, China}
\email{liujj@fudan.edu.cn}
\author{Yunfeng Shi}
\address{School of Mathematical Sciences, Fudan University, Shanghai 200433, China}
\email{yunfengshi13@fudan.edu.cn}
\author{Xiaoping Yuan}
\address{School of Mathematical Sciences, Fudan University, Shanghai 200433, China}
\email{xpyuan@fudan.edu.cn}
%\curraddr{Department of Mathematics and Statistics,
%Case Western Reserve University, Cleveland, Ohio 43403}




%    General info
\subjclass[2000]{Primary 37K55, 37J40; Secondary 35B35, 35Q35}

%\date{January 1, 2001 and, in revised form, June 22, 2001.}


\keywords{Stability for Hamilton PDEs; Almost periodic solution; full dimensional tori; NLS equation}

\begin{abstract}
In this paper, it is proved that the full dimensional invariant tori obtained by Bourgain [J. Funct. Anal., \textbf{229} (2005), no. 1, 62-94.] is stable in a very long time for 1D nonlinear Schr\"{o}dinger
equation with periodic boundary conditions.
\end{abstract}

\maketitle

%\section*{This is an unnumbered first-level section head}
%This is an example of an unnumbered first-level heading.

%% The correct journal style for \specialsection is all uppercase; a known bug
%% in amsart.cls prevents this, so input must be uppercase until it is fixed.
%\specialsection*{This is a Special Section Head}
%\specialsection*{THIS IS A SPECIAL SECTION HEAD}
%This is an example of a special section head%
%%%%%%%%%%%%%%%%%%%%%%%%%%%%%%%%%%%%%%%%%%%%%%%%%%%%%%%%%%%%%%%%%%%%%%%%
%\footnote{Here is an example of a footnote. Notice that this footnote
%text is running on so that it can stand as an example of how a footnote
%with separate paragraphs should be written.
%\par
%And here is the beginning of the second paragraph.}%
%%%%%%%%%%%%%%%%%%%%%%%%%%%%%%%%%%%%%%%%%%%%%%%%%%%%%%%%%%%%%%%%%%%%%%%%

\section{Introduction and main results}
Consider a Hamiltonian of $n$-freedom
\begin{equation}\label{yuan1}H=H_0(I)+\epsilon H_{1}(\theta,I),\end{equation}
with the standard symplectic  structure $\mathrm{d}\theta\wedge\mathrm{d}I$ on $\mathbb{T}^n\times\mathbb{R}^{n}$ and the angle-action variable $(\theta,I)$ belongs to some domain
$\mathbb{T}^n\times D\subseteq\mathbb{T}^n\times\mathbb{R}^{n}$. Assume the unperturbed Hamiltonian  $H_0(I)$ is independent of $\theta$ and satisfies Kolmogorov non-degenerate condition
$$\mathrm{det}(\partial^2H_0(I))\neq0,\ I\in D.$$
Also assume $H_0,H_1$ are smooth sufficiently. Then the well-known Kolmogorov-Arnold-Moser (KAM) theorem (\cite{K,A,M}) claims that any invariant tori of the unperturbed $H_0$ with prescribed Diophantine frequency $\omega(I_0)=\frac{\partial{H_0}(I_0)}{\partial I}$ for some $I_0\in D$ persist under a small perturbation $\epsilon H_1(I,\theta)$. This theorem is now called the classical KAM theorem and the persisted tori called full dimensional
KAM tori. The huge challenge is encountered when one tries to extend the classical KAM theorem to the Hamiltonian defined by some PDEs because of very complicated resonant relationships among the infinitely many number of frequencies. In order to evade this impasse, one considers the Hamiltonian of the form
$$H=N+\epsilon P(\theta,I,z,\bar{z}), $$
with the symplectic structure $\mathrm{d} \theta\wedge\mathrm{d} I+\sqrt{-1}\mathrm{d}z\wedge\mathrm{d}\bar{z}$
 on $\mathbb{T}^n\times\mathbb{R}^{n}\times\mathcal{H}\times\mathcal{H}\ni(\theta,I,z,\bar{z})$ and
$$N=\sum_{i=1}^{n}\omega_iI_i+\sum_{j=1}^{\tau}\Omega_{j}z_j\bar{z}_{j},\  1\leq n<\infty,1\leq\tau\leq \infty,$$
where $\mathcal{H}$ is a Hilbert space of dimension $\tau$, $\omega=(\omega_1,\omega_2,\cdots,\omega_n)$ called tangent frequency vector, $(\Omega_j)_{1\leq j\leq\tau}$ the normal frequency vector, and $P=P(\theta,I,z,\bar{z})$ is a perturbation. The unperturbed Hamiltonian $N$ has a special invariant torus
$$\mathcal{T}_0=\{\theta=\omega t\}\times\{I=0\}\times\{z=0\}\times\{\bar{z}=0\}.$$
Under suitable assumptions on $N$ and $P$, it can be proved that for ``most'' frequency $\omega$, the tori $\mathcal{T}_0$ can be persisted for some small perturbation $\epsilon P$ (see \cite{E1988} when $\tau<\infty$, see \cite{K3,K1} and \cite{W} when $\tau=\infty$). In \cite{K1,W}, the frequency vector $\omega\in\mathbb{R}^{n}$ is regarded as parameter. In \cite{E1988}, the frequency $$\omega=\omega_0t,$$ where $\omega_0\in\mathbb{R}^{n}$ is a fixed Diophantine vector and $t\in\mathbb{R}$ is regarded as parameter. Anyway, it is proved that a 1-dimensional parameter is needed in \cite{Bour2}, at least, when considering the lower dimensional KAM tori. Thus the work by \cite{E1988} is optimal in this sense, and the frequencies of those lower dimensional tori can not prescribed prior.
 Besides, the KAM theorem of this type depends heavily on the fact that the spatial dimension of the PDEs equals to 1. Bourgain \cite{B3,Bour6} developed a new method initialed by Craig-Wayne \cite{C-W} to deal with the KAM tori for the PDEs in high spatial dimension, based on the Newton iteration, Fr\"{o}hlich-Spencer techniques, Harmonic analysis and semi-algebraic set theory (see \cite{Bour6}). This method is now called C-W-B method. We also mention \cite{E-K} where the KAM theorem is extended in the direction of \cite{K1,K3,E1988,W} to deal with higher spatial dimensional nonlinear Schr\"{o}dinger equation. In addition, the KAM theory is also developed to deal  some 1-dimensional PDEs of unbounded perturbation. See, for example, \cite{K1}, \cite{KP}, \cite{LY1}, \cite{Zhang}, \cite{Baldi}, \cite{Baldi2}, \cite{Feola}, for the details. In the all above works, the obtained KAM tori are lower (finite) dimension, considering that the Hamiltonian PDEs are infinite dimensional.

Naturally, the following problem is interesting:
 \vskip8pt
 {\it Can the full dimensional KAM tori be expected with a suitable decay, for example, $I_n\sim|n|^{-S}$ with some $ S>0$ as $|n|\rightarrow+\infty$ ? }
\vskip8pt
  The existence of the  full dimensional KAM tori  with decay rate $I_n\sim|n|^{-S}$ is still open up to now. See \cite{Kuk3} for the details.  One way to obtain the existence of full dimensional KAM tori is to use repeatedly (infinitely many times) the KAM theorem dealing with lower dimensional KAM tori. See \cite{P2002} and \cite{B96} and some other references. However, the amplitude (or action) of those full dimensional KAM tori decays extremely fast. In fact, the decay rate is defined implicitly and much more fast than $I_n\sim e^{-|n|^{S}}, S>1$.  See more comments in \cite{Bour5}.
 Another way is due to Bourgain in \cite{Bour2005JFA}  where 1-dimensional nonlinear Schr\"{o}dinger equation with periodic boundary condition
is investigated (Also see \cite{P1990} given by P\"{o}schel where infinite dimensional Hamiltonian systems with short range is considered). It is shown in \cite{Bour2005JFA} that 1D NLS has a full dimensional KAM torus of prescribed frequencies  with the actions of the tori obeying the estimates \begin{equation}\label{083101}\frac{1}{2}e^{-r{|n|^{1/2}}}\leq I_n\leq 2e^{-r{|n|^{1/2}}}, n\in\mathbb{Z},\ r>0.
\end{equation}
This is up to now only one existence result about the full dimensional KAM tori with a slower decay rate than $I_n\sim e^{-|n|^{S}}, S>1$.

\iffalse %%%%%%%%
%%%%%%%%%
%%%%%%%%%%%%
 In order to fulfill the decay rate, Bourgain introduced a weight function $\sum_n(2a_n+k_n+k_n'){|n|^{1/2}}$ for a polynomial Hamiltonian
\[\sum_{a,k,k^\prime} B_{a,k,k^\prime} \prod_{n} I_n^{a_n} q_n^{k_n}{\bar q}_n^{k^\prime_n}. \]
In Remark 2 (p. 67, \cite{Bour2005JFA}), Bourgain stated that

\vskip8pt

{\it  the weight function $\sum_n(2a_n+k_n+k_n'){|n|^{1/2}}$ may have been replaced by expression $\sum_n(2a_n+k_n+k_n)|n|^{\theta}$ for some $0<\theta<1$}.

\vskip8pt

In the present paper, one of our main aims is to prove the Bourgain's statement. Thus, it will be shown that there is  a full dimensional KAM tori of prescribed frequencies for the NLS with decay rate of \begin{equation}\label{083102}\frac{1}{2}e^{-r{|n|^{\theta}}}\leq I_n\leq 2e^{-r{|n|^{\theta}}}, n\in\mathbb{Z},\ r>0, \ 0<\theta<1.
\end{equation}

\fi%%%%
%%%%%%%%%
%%%%%%%%

\iffalse%%%(also see P\"{o}schel \cite{P1990} where infinite dimensional Hamiltonian systems with short range is considered)
mmmmm
\fi

 On the other hand, it is well-known that  a physical quantity is observable only if it is stable at least for a long time. Naturally, one has the following question:

 \vskip8pt

{\it  Are the full dimensional KAM tori  obtained by Bourgain in \cite{Bour2005JFA}  stable for a long time?}

\vskip8pt

There have been a relatively long history about the long time stability for  the finite dimensional Hamiltonian (\ref{yuan1}) of freedom $n<\infty$. If the unperturbed $H_0$ is convex ( the steepest, in Nekhoroshev's terminology), any solutions including the KAM tori are stable in long time $|t|<\exp(\epsilon^{-\frac{1}{2n}})$, by using Nekhoroshev'e estimate \cite{Nekhoroshev},\cite{Poschel-Nekh}. In fact, those KAM tori are stable in a much longer time $\exp(\exp(\epsilon^{-\frac{1}{a}}))$ with some $a>n$ (see \cite{Morbid-G} for the details). Clearly, the stability of this kind can not be generalized to the Hamiltonian PDEs including NLS, in view of $n=\infty$ at this case. Bambusi \cite{BG93} and Bourgain \cite{B6} initiated the study of the stability in long time $|t|<\epsilon^{-M}$ with a large $M>0$ for the equilibrium $u=0$ for some Hamiltonian PDEs including NLS. See \cite{Bam1,Bam3,BG93,BG,BDGS,BN,BFG,BerB,B6,Delort,DS,FG,GIP,YZ}, for example, for more results. Recently, \cite{CLY} and \cite{CGL}  investigated the stability in long time for the lower (finite) dimensional KAM tori for PDEs.

According to our best knowledge, there has not yet been any result with respect to the stability in long time for the full dimensional KAM tori for the Hamiltonian PDEs. The main aim of the present paper is to prove that the  full dimensional KAM tori obtained by Bourgain  are stable in a long time. Incidentally, we will also prove that those tori are linearly stable.

 In order to state our theorem, let us begin with  the nonlinear Schr\"{o}dinger
equation with periodic boundary conditions
\begin{equation}\label{L1}
\mathbf{i}u_t-u_{xx}+Mu+\epsilon|u|^4u=0,\hspace{12pt}x\in\mathbb{T},
\end{equation}
where $M$ is a random Fourier multiplier defined by
\begin{equation}\label{L2}
\widehat{Mu}(n)=V_n\widehat{u}(n)
\end{equation}
and $(V_n)_{n\in\mathbb{Z}}$ are independently chosen in $[-1,1]$. Written in Fourier modes $(q_n)_{n\in\mathbb{Z}}$, then
(\ref{L1}) can be rewritten as
\begin{equation}\label{L3}
\dot{q}_n=\mathbf{i}\frac{\partial{H}}{\partial\bar{q}_n}
\end{equation}
with the Hamiltonian
\begin{equation}\label{L4}
H(q,\bar{q})=\sum_{n\in\mathbb{Z}}(n^2+V_n)|q_n|^2+\epsilon\sum_{n_1-n_2+n_3-n_4+n_5-n_6=0}q_{n_1}\bar{q}_{n_2}q_{n_3}\bar{q}_{n_4}q_{n_5}\bar{q}_{n_6}.
\end{equation}
Fix $0<\theta<1$ and introduce for any $r>0$ the Banach space $\mathfrak{H}_{r,\infty}$ of all complex-valued sequences
$q=(q_n)_{n\in\mathbb{Z}}$ with
\begin{equation}\label{042501}
\|q\|_{r,\infty}=\sup_{n\in\mathbb{Z}}|q_n|e^{r|n|^{\theta}}<\infty.
\end{equation}
For $x\in\mathbb{R}$, denote $\|x\|=\inf\limits_{y\in\mathbb{Z}}|x-y|$. Then we say a vector $V\in[-1,1]^{\mathbb{Z}}$ is Diophantine, if there exists a real number $\gamma>0$ such that
\begin{equation}\label{S1}
\left|\left|\sum_{n\in\mathbb{Z}}l_nV_n\right|\right|\geq \gamma\prod_{n\in\mathbb{Z}}\frac{1}{1+l_n^2|n|^4},
\end{equation}
\mbox{for any $l\in\mathbb{Z}^{\mathbb{Z}}$ with $0<\#\mathrm{supp}\ l<\infty$}, where
\begin{equation*}
\mathrm{supp}\ l=\left\{n:l_n\neq0\right\}
\end{equation*}
and
\begin{equation*}
|n|=\max\{1,n,-n\}.
\end{equation*}From Lemma 4.1 in \cite{Bour2005JFA}, we know the following resonance issue:
\begin{equation}\label{S3}
P\left\{V: \left|\left|\sum_{n\in\mathbb{Z}}l_nV_n\right|\right|< \gamma\prod_{n\in\mathbb{Z}}\frac{1}{1+l_n^2|n|^4},\forall\  l\neq0 \ \mathrm{with}\ \#\mathrm{supp}\ l<\infty\right\}<C\gamma,
\end{equation}
where $P$ is the standard probability measure on $[-1,1]^{\mathbb{Z}}$ and $C>0$ is an absolute constant.
\begin{theorem}\label{L10}
Given $0<\theta<1$, $r>0$ and a Diophantine vector $\omega=(\omega_n)_{n\in\mathbb{Z}}$ satisfying $\sup_{n}|\omega_n|<1$, then for sufficiently small $\epsilon>0$ and appropriate $M$, (\ref{L1}) has a full dimensional invariant torus $\mathfrak{T}$ satisfying:
\begin{itemize}
 \item[(1)]the amplitude on $\mathfrak{T}$ is restricted as
 \begin{equation}\label{L11}
 \frac{1}{2}e^{-r|n|^{\theta}}\leq|q_n|\leq e^{-r|n|^{\theta}};
 \end{equation}
 \item[(2)]the frequency on $\mathfrak{T}$ is prescribed to be $(n^2+\omega_n)_{n\in\mathbb{Z}}$;
     \item[(3)] the invariant tori $\mathfrak{T}$ are linearly stable;
     \item[(4)] the invariant tori are stable in a long time in the sense that for any small enough $\tau$ (independent of $\epsilon$), if
         \begin{equation*}
         d(q(0),\mathfrak{T})\leq \tau,
         \end{equation*}
         then
         \begin{equation}\label{081101}
         d(q(t),\mathfrak{T})\leq 2\tau,\qquad \mbox{for all $|t|\leq \tau ^{-\frac14|\ln \tau|^{\frac{\theta}{10}}+1}$},
         \end{equation}
         where
         \begin{equation*}
         d(q(0),\mathfrak{T})=\inf_{w\in\mathfrak{T}}||q(0)-w||_{r,\infty}.
         \end{equation*}
\end{itemize}
\end{theorem}

\begin{remark}

In order to fulfill the decay rate, Bourgain introduced a weight function $\sum_n(2a_n+k_n+k_n'){|n|^{1/2}}$ for a polynomial Hamiltonian
\[\sum_{a,k,k^\prime} B_{akk^\prime} \prod_{n} I_n^{a_n} q_n^{k_n}{\bar q}_n^{k^\prime_n}. \]
In Remark 2 (p. 67, \cite{Bour2005JFA}), Bourgain stated that  the weight function $\sum_n(2a_n+k_n+k_n'){|n|^{1/2}}$ may have been replaced by expression $\sum_n(2a_n+k_n+k_n)|n|^{\theta}$ for some $0<\theta<1$.
In the present paper, we fulfill the Bourgain's statement incidentally.
\end{remark}

%%%%%%%%%%%%%%%%%%%%%

\section{The Norm of the Hamiltonian}
Most of the notations come from \cite{Bour2005JFA} for the reader's easy understanding. The analysis will be performed in complex conjugate variables $(q_n,\bar q_n)$ without passing to action-angle variables. The Hamiltonian expressions may involve $I_n=|q_n|^2$ and $J_n=I_n-I_n(0)$ as notations but not as new variables, where $I_n(0)$ will be considered as the initial data. At every stage of the iteration, the Hamiltonian
$H$ will be expanded in monomials $\mathcal{M}_{akk'}$ ($a,k,k'\in\mathbb{N}^{\mathbb{Z}}$ are multi-indices) of the following form:
\begin{equation}
\mathcal{M}_{akk'}=\prod_{n\in\mathbb{Z}}I_n(0)^{a_n}q_n^{k_n}\bar q_n^{k_n'},
\end{equation}
where
\begin{equation}
\sum_{n\in\mathbb{Z}}k_n=\sum_{n\in\mathbb{Z}}k_n',
\end{equation}
and
\begin{equation}\label{050901}
\sum_{n\in\mathbb{Z}} nk_n=\sum_{n\in\mathbb{Z}}nk_n'.
\end{equation}
Define by
\begin{equation}
\mbox{supp}\ \mathcal{M}_{akk'}=\{n:2a_n+k_n+k_n'\neq 0\},
\end{equation}
and
\begin{equation}
\mbox{degree}\ \mathcal{M}_{akk'}=\sum_{n\in\mathbb{Z}}(2a_n+k_n+k_n').
\end{equation}
With these notations, the  Hamiltonian (\ref{L4}) has the form of
\begin{equation}\label{H}
H(q,\bar q)=H_{2}(q,\bar q)+\sum_{a,k,k'}B_{akk'}\mathcal{M}_{akk'}
\end{equation}
where
\begin{equation*}
H_{2}(q,\bar q)=\sum_{n\in\mathbb{Z}}(n^2+V_n)|q_n|^2,
\end{equation*}
$B_{akk'}$ are the coefficients, $V_n\in[-1,1]\ \mbox{for}\ \forall\  n\in\mathbb{Z}$ and assuming $$\sum_{n\in\mathbb{Z}}(2a_n+k_n+k'_n)=6.$$

Denote by
\begin{equation*}
n_1^*=\max\{|n|:a_n+k_n+k_n'\neq 0\}.
\end{equation*}
Before defining the norm of the Hamiltonian, we give the following lemmas:
\begin{lemma}\label{005}
Denote $(n^*_i)_{i\geq1}$ the decreasing rearrangement of
\begin{equation*}
\{|n|:\ \mbox{where $n$ is repeated}\ 2a_n+k_n+k_n'\ \mbox{times}\},
\end{equation*}
and assume
\begin{equation}\label{0m}\sum_{n\in\mathbb{Z}}(k_n-k'_n)n=0.\end{equation}
Then for any $0<\theta<1$, one has
\begin{equation}\label{001}
\sum_{n\in\mathbb{Z}}(2a_n+k_n+k_n')|n|^{\theta}\geq2(n_1^*)^{\theta}+(2-2^{\theta})\sum_{i\geq 3}(n_i^*)^{\theta}.
\end{equation}
\end{lemma}
\begin{proof}Without loss of generality,
denote $(n_i),\ |n_1|\geq |n_2|\geq\cdots$, the system $\{n:\ \mbox{where $n$ is repeated}\ 2a_n+k_n+k_n'\ \mbox{times}\}$ and we have $n_i^*=|n_i|\ \mbox{for}\ \forall\ i\geq1$. In view of (\ref{0m}), there exist $(\mu_i)_{i\geq1}$ with $\mu_i\in\{\pm1\}$ such that
\begin{equation*}
\sum_{i\geq1}\mu_in_i=0,
\end{equation*}
and hence
\begin{equation*}
n_1^*\leq \sum_{i\geq 2}|n_i|.
\end{equation*}
Consequently
\begin{equation*}
(n_1^*)^{\theta}\leq\left(\sum_{i\geq 2}|n_i|\right)^{\theta}.
\end{equation*}
Thus the inequality (\ref{001}) will follow from the inequality
\begin{equation}\label{002}
\sum_{i\geq 2}|n_i|^{\theta}\geq \left(\sum_{i\geq 2}|n_i|\right)^{\theta}+(2-2^{\theta})\sum_{i\geq 3}|n_i|^\theta.
\end{equation}
To prove the inequality (\ref{002}), one just needs the following fact: consider the function
\begin{equation*}
f(x)=(1+x)^{\theta}-x^{\theta},\qquad x\in[1,\infty),
\end{equation*}
and one has
\begin{equation}\label{003}
\max_{x\in[1,\infty)}f(x)=f(1)=2^{\theta}-1,
\end{equation}
which is based on
\begin{equation*}
f'(x)=\theta((1+x)^{\theta-1}-x^{\theta-1})<0,\  \mbox{for}\ x\in[1,\infty)\ \mbox{and}\ \forall\  \theta\in(0,1).
\end{equation*}
Hence, for any $a\geq b>0$, we have
\begin{eqnarray}
\nonumber\left(a+b\right)^{\theta}+(2-2^{\theta})b^{\theta}-a^{\theta}-b^{\theta}
&=&\nonumber\left(a+b\right)^{\theta}-a^{\theta}+(1-2^{\theta})b^{\theta}\\
&=&\nonumber b^{\theta}\left(\left(1+\frac{a}{b}\right)^{\theta}-\left(\frac{a}{b}\right)^{\theta}-(2^{\theta}-1)\right)\\
\nonumber&\leq&0 \ \ \mbox{(in view of (\ref{003}))},
\end{eqnarray}
that is
\begin{equation}\label{042505}
a^{\theta}+b^{\theta}\geq\left(a+b\right)^{\theta}+(2-2^{\theta})b^{\theta}\ \ \mbox{(for $\forall\ a\geq b>0$)}.
\end{equation}
By iteration, one obtains
\begin{eqnarray*}
\sum_{i\geq 2}|n_i|^{\theta}
&=&|n_2|^{\theta}+|n_3|^{\theta}+\sum_{i\geq 4}|n_i|^{\theta}\\
&\geq&\left(|n_2|+|n_3|\right)^{\theta}+(2-2^{\theta})|n_3|^{\theta}+\sum_{i\geq 4}|n_i|^{\theta}\qquad (\mbox{in view of (\ref{042505})})\\
&=&\left(|n_2|+|n_3|\right)^{\theta}+|n_4|^{\theta}+\sum_{i\geq 5}|n_i|^{\theta}+(2-2^{\theta})|n_3|^{\theta}\\
&\geq&\left(|n_2|+|n_3|+|n_4|\right)^{\theta}+(2-2^{\theta})|n_4|^{\theta}+\sum_{i\geq 5}|n_i|^{\theta}+(2-2^{\theta})|n_3|^{\theta}\\
&&\qquad (\mbox{in view of (\ref{042505}) again})\\
&=&\left(|n_2|+|n_3|+|n_4|\right)^{\theta}+\sum_{i\geq 5}|n_i|^{\theta}+(2-2^{\theta})(|n_3|^{\theta}+|n_4|^{\theta})\\
&&\cdots\\
&\geq&\left(\sum_{i\geq 2}|n_i|\right)^{\theta}+(2-2^{\theta})\left(\sum_{i\geq 3}|n_i|^{\theta}\right).
\end{eqnarray*}
\end{proof}

Now we will define the norm of the Hamiltonian $\sum_{a,k,k'}B_{akk'}\mathcal{M}_{akk'}$ with the weight $\rho>0$ by
\begin{definition}\label{083103}
\begin{equation}\label{042602}
||H||_{\rho}=\max_{a,k,k'}\frac{|B_{akk'}|}{e^{\rho\sum_{n}((2a_n+k_n+k_n')|n|^{\theta}-2(n_1^*)^{\theta})}}.
\end{equation}
\end{definition}

\section{The Homological Equations}
\subsection{Derivation of homological equations}
The proof of Main Theorem employs the rapidly
converging iteration scheme of Newton type to deal with small divisor problems
introduced by Kolmogorov, involving the infinite sequence of coordinate transformations.
At the $s$-th step of the scheme, a Hamiltonian
$H_{s} = N_{s} + R_{s}$
is considered, as a small perturbation of some normal form $N_{s}$ . A transformation $\Phi_{s}$ is
set up so that
$$H_{s}\circ \Phi_{s} = N_{s+1} + R_{s+1}$$
with another normal form $N_{s+1}$ and a much smaller perturbation $R_{s+1}$. We drop the index $s$ of $H_{s}, N_{s}, R_{s}, \Phi_{s}$ and shorten the index $s+1$ as $+$.

Now consider the Hamiltonian $H$ of the form
\begin{eqnarray}\label{N1}
{H}=N+R,
\end{eqnarray}
where
\begin{equation*}
N=\sum_{n\in\mathbb{Z}}(n^2+\widetilde V_n)|q_n|^2,
\end{equation*}
and
\begin{equation*}
R=R_0+R_1+R_2
\end{equation*}
with $|\widetilde V_n|\leq 2$ for all $n\in\mathbb{Z}$,
\begin{eqnarray*}
{R}_0&=&\sum_{a,k,k'\atop\mbox{supp}\ k\bigcap \mbox{supp}\ k'=\emptyset}B_{akk'}\mathcal{M}_{akk'},\\
{R}_1&=&\sum_{n\in\mathbb{Z}}J_n\left(\sum_{a,k,k'\atop\mbox{supp}\ k\bigcap \mbox{supp}\ k'=\emptyset}B_{akk'}^{(n)}\mathcal{M}_{akk'}\right),\\
{R}_2&=&\sum_{n_1,n_2\in\mathbb{Z}}J_{n_1}J_{n_2}\left(\sum_{a,k,k'\atop\mbox{no assumption}}B_{akk'}^{(n_1,n_2)}\mathcal{M}_{akk'}\right),
\end{eqnarray*}
and
\begin{equation*}
J_n=I_n-I_n(0),\quad I_n=|q_n|^2.
\end{equation*}
We desire to eliminate the terms $R_0,R_1$ in (\ref{N1}) by the coordinate transformation $\Phi$, which is obtained as the time-1 map $X_F^{t}|_{t=1}$ of a Hamiltonian
vector field $X_F$ with $F=F_0+F_1$. Let ${F}_{0}$ (resp.${F}_{1}$) has the form of ${R}_0$ (resp.${R}_{1}$),
that is \begin{eqnarray}
&&{F}_0=\sum_{a,k,k'\atop\mbox{supp}\ k\bigcap \mbox{supp}\ k'=\emptyset}F_{akk'}\mathcal{M}_{akk'},\\
&&{F}_1=\sum_{n\in\mathbb{Z}}J_n\left(\sum_{a,k,k'\atop\mbox{supp}\ k\bigcap \mbox{supp}\ k'=\emptyset}F_{akk'}^{(n)}\mathcal{M}_{akk'}\right),
\end{eqnarray}
and the homological equations become
\begin{equation}\label{4.27}
\{N,{F}\}+R_0+R_{1}=[R_0]+[R_1],
\end{equation}
where
\begin{equation}\label{051501}
[R_0]=\sum_{a}B_{a00}\mathcal{M}_{a00},
\end{equation}
and
\begin{equation}\label{051502}
[R_1]=\sum_{n\in\mathbb{Z}}J_n\sum_{a}B_{a00}^{(n)}\mathcal{M}_{a00}.
\end{equation}
The solutions of the homological equations (\ref{4.27}) are given by
\begin{equation}\label{051304}
F_{akk'}=\frac{B_{akk'}}{\sum_{n\in\mathbb{Z}}(k_n-k^{'}_n)(n^2+\widetilde{V}_n)},
\end{equation}
where
\begin{equation}\label{051305}
F_{akk'}^{(m)}=\frac{B_{akk'}^{(m)}}{\sum_{n\in\mathbb{Z}}(k_n-k^{'}_n)(n^2+\widetilde{V}_n)},
\end{equation}
and the new Hamiltonian ${H}_{+}$ has the form
\begin{eqnarray}
H_{+}\nonumber&=&H\circ\Phi\\
&=&\nonumber N+\{N,F\}+R_0+R_1\\
&&\nonumber+\int_{0}^1\{(1-t)\{N,F\}+R_0+R_1,F\}\circ X_F^{t}\ \mathrm{d}{t}
+\nonumber R_2\circ X_F^1\\
&=&\label{051401}N_++R_+,
\end{eqnarray}
where
\begin{equation}\label{051402}
N_+=N+[R_0]+[R_1],
\end{equation}
and
\begin{equation}\label{051403}
R_+=\int_{0}^1\{(1-t)\{N,F\}+R_0+R_1,F\}\circ X_F^{t}\ \mathrm{d} t+R_2\circ X_F^1.
\end{equation}
\subsection{The solutions of the homological equations}In this subsection, we will estimate
the solutions of the homological equations. To this end, we define a new norm for the Hamiltonian ${R}$ of the form as in (\ref{N1}) as follows:
\begin{equation}
||{R}||_{\rho}^{+}=\max\{||R_0||_\rho^{+},||R_1||_\rho^{+}|,||R_2||_\rho^{+}\},
\end{equation}
where
\begin{eqnarray}
&&\label{051302}||R_0||_\rho^{+}=\sup_{a,k,k'}\frac{|B_{akk'}|}{e^{\rho(\sum_{n}(2a_n+k_n+k_n')|n|^{\theta}-2(n_1^*)^{\theta})}},\\
&&||R_1||_\rho^{+}=\sup_{a,k,k'\atop m\in\mathbb{Z}}\frac{|B^{(m)}_{akk'}|}{e^{\rho(\sum_{n}(2a_n+k_n+k_n')|n|^{\theta}+2|m|^{\theta}-2(n_1^*)^{\theta})}},\\
&&||R_2||_\rho^{+}=\sup_{a,k,k'\atop
m_1,m_2\in\mathbb{Z}}\frac{|B^{(m_1,m_2)}_{akk'}|}{e^{\rho(\sum_{n}(2a_n+k_n+k_n')|n|^{\theta}
+2|m_1|^{\theta}+2|m_2|^{\theta}-2(n_1^*)^{\theta})}}.
\end{eqnarray}
Moreover, one has the following estimates:
\begin{lemma}\label{051301}
Given any $\rho,\delta>0$ and a Hamiltonian $R$, one has
\begin{equation}\label{N6}
||R||_{\rho+\delta}^{+}\leq\left(\frac{1}{\delta}\right)^{ C(\theta)\delta^{-\frac{1}{\theta}}}||R||_{\rho}
\end{equation}
and
\begin{equation}\label{N7}
||R||_{\rho+\delta}\leq\frac{C(\theta)}{\delta^2}||R||_{\rho}^{+},
\end{equation}
where $C(\theta)$ is a  positive constant depending on $\theta$ only.
\end{lemma}
\begin{proof}
See the details of the proof in the Appendix.
\end{proof}

\begin{lemma}\label{S6}
Given $\theta\in(0,1)$, let $(\widetilde{V}_n)$ be Diophantine with $\gamma>0$ (see (\ref{S1})). Then for any $\rho>0,0<\delta\ll1$ (depending only on $\theta$), the solutions of the homological equations (\ref{4.27}), which are given by (\ref{051304}) and (\ref{051305}), satisfy
\begin{eqnarray}\label{S7}
||{F}_i||_{\rho+\delta}^{+}\leq \frac{1}{\gamma}\cdot e^{C(\theta)\delta^{-\frac5\theta}}||{{R_i}}||_{\rho}^{+},
\end{eqnarray}
 where $i=0,1$ and $C(\theta)$ is a positive constant depending on $\theta$ only.
\end{lemma}

\begin{proof}

We distinguish two cases:

$\textbf{Case. 1.}$ $$\left|\sum_{n\in\mathbb{Z}}(k_n-k'_n)n^2\right|>10\sum_{n\in\mathbb{Z}}|k_n-k'_n|.$$

Since $|\widetilde{V}_n|\leq2$, we have
$$\left|\sum_{n\in\mathbb{Z}}(k_n-k'_n)(n^2+\widetilde{V}_n)\right|>10\sum_{n\in\mathbb{Z}}|k_n-k'_n|-2\sum_{n\in\mathbb{Z}}|k_n-k'_n|\geq1,$$
where the last inequality is based on $\mbox{supp}\ k\bigcap \mbox{supp}\ k'=\emptyset$.
There is no small divisor and (\ref{S7}) holds trivially.

$\textbf{Case. 2.}$ $$\left|\sum_{n\in\mathbb{Z}}(k_n-k'_n)n^2\right|\leq 10\sum_{n\in\mathbb{Z}}|k_n-k'_n|.$$

In this case, we always assume  $$\left|\sum_{n\in\mathbb{Z}}(k_n-k_n')(n^2+\widetilde V_n)\right|\leq1,$$otherwise there is no small divisor.

Firstly, one has
\begin{eqnarray}\label{S8}
\nonumber\sum_{n\in\mathbb{Z}}|k_n-k_n'||n|^{\theta/2}
&\leq&\nonumber{3\cdot 6^{\theta/2}}\left(\sum_{i\geq3}(n_i^*)^{\theta}\right)\qquad (\mbox{in view of Lemma \ref{a1}})\\
&\leq&\frac{3\cdot 6^{\theta/2}}{2-2^{\theta}}\left(\sum_{n\in\mathbb{Z}}(2a_n+k_n+k_n')|n|^{\theta}-2(n_1^*)^{\theta}\right),
\end{eqnarray}
where the last inequality is based on Lemma \ref{005}.

Since $$\sum_{n\in\mathbb{Z}}(k_n-k'_n)n^2\in\mathbb{Z},$$
the Diophantine property of $(\widetilde{V}_n)$ implies
\begin{equation}\label{S9}
\left|\sum_{n\in\mathbb{Z}}(k_n-k'_n)(n^2+\widetilde{V}_n)\right|\geq\frac\gamma2\prod_{n\in\mathbb{Z}}\frac{1}{1+{|k_n-k'_n|}^2|n|^4}.
\end{equation}
Hence,
\begin{eqnarray}
\nonumber&&{|{F}_{akk'}|}e^{-(\rho+\delta)(\sum_{n}(2a_n+k_n+k'_n)|n|^{\theta}-2(n_1^{*})^{\theta})}\\
&=&\nonumber\frac{|{B}_{akk'}|}{|\sum_{n}(k_n-k'_n)(n^2+\widetilde{V}_n)|}e^{-(\rho+\delta)(\sum_{n}(2a_n+k_n+k'_n)|n|^{\theta}-2(n_1^{*})^{\theta})}\\
&&\nonumber\mbox{(in view of (\ref{051304}))}\\
&=&\nonumber|{B}_{akk'}|e^{-\rho(\sum_{n}(2a_n+k_n+k'_n)|n|^{\theta}-2(n_1^{*})^{\theta})}
\\
&&\nonumber\times \frac{e^{-\delta(\sum_{n}(2a_n+k_n+k'_n)|n|^{\theta}-2(n_1^{*})^{\theta})}}{|\sum_{n}(k_n-k'_n)(n^2+\widetilde{V}_n)|}\\
\nonumber&\leq& 2\gamma^{-1}||{R_0}||_{\rho}^{+} \prod_{n}\left({1+{|k_n-k'_n|}^2|n|^4}\right)e^{-\delta\left(\sum_{n}(2a_n+k_n+k'_n)|n|^{\theta}-2(n_1^*)^{\theta}\right)}\\
\nonumber&&
 \mbox{(in view of (\ref{051302}) and (\ref{S9}))}\\
\nonumber&\leq&2\gamma^{-1} ||{R_0}||_{\rho}^+e^{\sum_{n}\ln(1+|k_n-k'_n|^2|n|^4)}e^{-\frac{3\cdot 6^{\theta/2}\delta}{2-2^{\theta}}\sum_n\left(|k_n-k_n'||n|^{\theta/2}\right)}\\
&&\nonumber\mbox{(in view of (\ref{S8}))}\\
\nonumber&\leq& 2\gamma^{-1} ||{R_0}||_{\rho}^+e^{\sum_{n}\ln(1+|k_n-k'_n|^2|n|^4)}e^{-\delta \sum_n\left(|k_n-k_n'||n|^{\theta/2}\right)}\\
&&\nonumber \mbox{(in view of $\frac{3\cdot 6^{\theta/2}}{2-2^{\theta}}>1$)}\\
&=&\nonumber2\gamma^{-1} ||{R_0}||_{\rho}^+e^{\sum_{{n:k_n\neq k'_n}}\ln(1+|k_n-k'_n|^2|n|^4)-\delta\sum_{n:k_n\neq k'_n}\left(|k_n-k_n'||n|^{\theta/2}\right)}\\
&\leq&\nonumber2\gamma^{-1} ||{R_0}||_{\rho}^+e^{8\left(\sum_{{n:k_n\neq k'_n}}\ln(|k_n-k'_n||n|)\right)+3-\delta\sum_{n:k_n\neq k'_n}\left(|k_n-k_n'|^{\theta/2}|n|^{\theta/2}\right)}\\
&&\mbox{(in view of $0<\theta<1$)}\nonumber\\
&=&\nonumber \frac{2e^{3}}{\gamma} ||{R_0}||_{\rho}^+e^{\sum_{n:k_n\neq k'_n}\left(8\ln(|k_n-k'_n||n|)-\delta|k_n-k_n'|^{\theta/2}|n|^{\theta/2}\right)}\\
&=&\nonumber\frac{2e^{3}}{\gamma} ||{R_0}||_{\rho}^+e^{\sum_{|n|\leq N:k_n\neq k'_n}\left(8\ln(|k_n-k'_n||n|)-\delta|k_n-k_n'|^{\theta/2}|n|^{\theta/2}\right)}\\
&&+\nonumber\frac{2e^{3}}{\gamma} ||{R_0}||_{\rho}^+e^{\sum_{n>N:k_n\neq k'_n}\left(8\ln(|k_n-k'_n||n|)-\delta|k_n-k_n'|^{\theta/2}|n|^{\theta/2}\right)}\\
&&\nonumber \mbox{(where $N=\left(\frac{16}{\theta\delta}\right)^{4/\theta}$)}\\
&=&\nonumber\frac{2e^{3}}{\gamma} ||{R_0}||_{\rho}^+e^{\left(\frac{16}{\theta\delta}\right)^{4/\theta}\cdot \frac{16}{\theta}\ln\left(\frac{16}{\theta\delta}\right)}\qquad{\mbox{(in view of  (\ref{051201}) below)}}\\
&&+\nonumber\frac{2e^{3}}{\gamma} ||{R_0}||_{\rho}^+\qquad\qquad\qquad\qquad\mbox{(in view of (\ref{051202}) below)}\\
&\leq&\label{051306}\frac{1}{\gamma}\cdot e^{C(\theta)\delta^{-\frac5\theta}} ||{R_0}||_{\rho}^+ \  \ \mbox {(for $0<\delta\ll1$)},
\end{eqnarray}
where $C(\theta)$ is a positive constant depending on $\theta$ only.

Therefore, in view of (\ref{051302}) and (\ref{051306}), we finish the proof of (\ref{S7}) for $i=0$.

It is easy to verify the following two facts that
\begin{equation}\label{051201}
\max_{x\geq 1} f(x)=f\left(\left(\frac{16}{\theta\delta}\right)^{2/\theta}\right)=-\frac {16}{\theta} +8\ln\left(\left(\frac{16}{\theta\delta}\right)^{2/\theta}\right)\leq \frac{16}{\theta}\ln\left(\frac{16}{\theta\delta}\right)
\end{equation}
with $f(x)=(-\delta x^{\theta/2}+8\ln x)$,
and when $|n|>N=\left(\frac{16}{\theta\delta}\right)^{4/\theta}, k_n\neq k'_n$, one has
\begin{equation}\label{051202}
-\delta\left(|k_n-k_n'|^{\theta/2}|n|^{\theta/2}\right)+8\ln(|k_n-k'_n||n|)<0\ \  \mbox {(for $0<\delta\ll1$)}.
\end{equation}
Similarly, one can prove (\ref{S7}) for $i=1$.
\end{proof}

\section{The new Hamiltonian}
In view of (\ref{051401}), we obtain the new Hamiltonian
\begin{equation}
H_+=H\circ \Phi=N_++R_+,
\end{equation}
where $N_+$ and $R_+$ are given in (\ref{051402}) and (\ref{051403}) respectively.
\subsection{Estimating Poisson Bracket and Symplectic Transformation}
To estimate the Hamiltonian $H_+$, we need the following two lemmas.
\begin{lemma}\label{010}
Let $\theta\in(0,1),\rho>0$ and $0<\delta_1,\delta_2\ll1$ (depending on $\theta,\rho$). Then one has
\begin{equation}\label{042704}
||\{H_1,H_2\}||_\rho\leq \frac{1}{\delta_2}\left(\frac{1}{\delta_1}\right)^{C({\theta}){\delta_1^{-\frac{1}{\theta}}}}||H_1||_{\rho-\delta_1}||H_2||_{\rho-\delta_2},
\end{equation}
where $C(\theta)$ is a positive constant depending on $\theta$ only.
\end{lemma}

\begin{proof}
Let
\begin{equation*}
H_1=\sum_{a,k,k'} b_{akk'}\mathcal{M}_{akk'}
\end{equation*}
and
\begin{equation*}
H_2=\sum_{A,K,K'} B_{AKK'}\mathcal{M}_{AKK'}.
\end{equation*}
Hence
\begin{equation*}
\{H_1,H_2\}=\sum_{a,k,k',A,K,K'}b_{akk'}B_{AKK'}\{\mathcal{M}_{akk'},\mathcal{M}_{AKK'}\}.
\end{equation*}
Write
\begin{eqnarray*}
\{\mathcal{M}_{akk'},\mathcal{M}_{AKK'}\}
&=&\frac{1}{2\textbf{i}}\sum_j\left(\prod_{n\neq j}I_n(0)^{a_n+A_n}q_n^{k_n+K_n}\bar{q}_n^{k_n'+K_n'}\right)\\
&&\times\left((k_jK_j'-k_j'K_j)I_j(0)^{a_j+A_j}q_j^{k_j+K_j-1}\bar{q}_j^{k_j'+K_j'-1}\right).
\end{eqnarray*}
Then the coefficient of $\prod_{n}I_n(0)^{\alpha_n}q_n^{\kappa_n}\bar{q}_n^{\kappa'_n}=\mathcal{M}_{\alpha\kappa\kappa'}$ is given by
\begin{equation}\label{006}
2\textbf{i} B_{\alpha\kappa\kappa'}=\sum_{j}\sum_{*}\sum_{**}(k_jK_j'-k_j'K_j)b_{akk'}B_{AKK'},
\end{equation}
where
\begin{equation*}
\sum_{*}=\sum_{a,A \atop a+A=\alpha},
\end{equation*}
and
\begin{equation*}
\sum_{**}=\sum_{k,k',K,K'\atop \mbox{when}\ n\neq j, k_n+K_n=\kappa_n,k_n'+K_n'=\kappa_n';\mbox{when}\ n=j, k_n+K_n-1=\kappa_n,k_n'+K_n'-1=\kappa_n'}.
\end{equation*}
\begin{remark}
Note that
\begin{equation}\label{052802}
\sum_n(2\alpha_n+\kappa_n+\kappa_n')=\sum_n(2a_n+k_n+k_n')+\sum_n(2A_n+K_n+K_n')-2
\end{equation}
and
\begin{equation}\label{052801}
\sum_n(2\alpha_n+\kappa_n+\kappa_n')|n|=\sum_n(2a_n+k_n+k_n')|n|+\sum_n(2A_n+K_n+K_n')|n|-2|j|.
\end{equation}
\end{remark}
In view of (\ref{042602}), one has
\begin{eqnarray}
\nonumber |b_{akk'}|
 \nonumber &\leq& ||H_1||_{\rho-\delta_1}e^{(\rho-\delta_1)\sum_{n}(2a_n+k_n+k_n')|n|^{\theta}-2(\rho-\delta_1)(n_1^*)^{\theta}}\\
&=&||H_1||_{\rho-\delta_1}e^{\rho\sum_{n}(2a_n+k_n+k_n')|n|^{\theta}-2\rho(n_1^*)^{\theta}}
e^{-\delta_1\sum_{n}(2a_n+k_n+k_n')|n|^{\theta}+2\delta_1(n_1^*)^{\theta}}\nonumber\\
\label{007}&\leq&||H_1||_{\rho-\delta_1}e^{\rho\sum_{n}(2a_n+k_n+k_n')|n|^{\theta}-2\rho(n_1^*)^{\theta}}e^{-(2-2^{\theta})\delta_1
\sum_{i\geq 3}(n_i^*)^{\theta}},
\end{eqnarray}
where the last inequality is based on Lemma \ref{005}.

Similarly,
\begin{equation}
|B_{AKK'}|\leq \label{008}||H_2||_{\rho-\delta_2}e^{\rho\sum_{n}(2A_n+K_n+K_n')|n|^{\theta}-2\rho(N_1^*)^{\theta}}e^{-(2-2^{\theta})\delta_2\sum_{i\geq 3}(N_i^*)^{\theta}}.
\end{equation}
Substitution of (\ref{007}) and (\ref{008}) in (\ref{006}) gives
\begin{eqnarray*}
\label{009}|2\textbf{i}B_{\alpha\kappa\kappa'}|
\nonumber&\leq&||H_1||_{\rho-\delta_1}||H_2||_{\rho-\delta_2}\sum_{j}\sum_{*}\sum_{**}|k_jK_j'-k_j'K_j|\\
&&\times e^{\rho\left(\sum_{n}(2a_n+k_n+k_n')|n|^{\theta}-2(n_1^*)^{\theta}+\sum_{n}(2A_n+K_n+K_n')|n|^{\theta}-2(N_1^*)^{\theta}\right)}\\
\nonumber&&\times e^{-(2-2^{\theta})\delta_1\sum_{i\geq 3}(n_i^*)^{\theta}}e^{-(2-2^{\theta})\delta_2\sum_{i\geq 3}(N_i^*)^{\theta}}\\
\nonumber&=&||H_1||_{\rho-\delta_1}||H_2||_{\rho-\delta_2}\sum_{j}\sum_{*}\sum_{**}|k_jK_j'-k_j'K_j|\\
&&\times e^{\rho\left(\sum_{n}(2\alpha_n+\kappa_n+\kappa_n')|n|^{\theta}+2|j|^{\theta}\right)}e^{-2\rho(n_1^*)^{\theta}-2\rho(N_1^*)^{\theta}}\\
\nonumber&&\times  e^{-(2-2^{\theta})\delta_1\sum_{i\geq 3}(n_i^*)^{\theta}}e^{-(2-2^{\theta})\delta_2\sum_{i\geq 3}(N_i^*)^{\theta}}\\
&&\nonumber(\mbox{in view of (\ref{052801})})\\
\nonumber&=&||H_1||_{\rho-\delta_1}||H_2||_{\rho-\delta_2}e^{\rho(\sum_{n}(2\alpha_n+\kappa_n+\kappa_n')|n|^{\theta}-2(\nu_1^*)^{\theta})}\\
&&\times\sum_{j}\sum_{*}\sum_{**}|k_jK_j'-k_j'K_j| e^{2\rho(|j|^{\theta}+(\nu_1^*)^{\theta}-(n_1^*)^{\theta}-(N_1^*)^{\theta})}\\
&&\times e^{-(2-2^{\theta})\delta_1\sum_{i\geq 3}(n_i^*)^{\theta}}e^{-(2-2^{\theta})\delta_2\sum_{i\geq 3}(N_i^*)^{\theta}},
\end{eqnarray*}
where
\begin{equation*}
\nu_1^*=\max\{|n|:\alpha_n+\kappa_n+\kappa_n'\neq0\}.
\end{equation*}
To show (\ref{042704}) holds, it suffices to prove
\begin{equation}\label{041802}
I\leq \frac{1}{\delta_2}\left(\frac{1}{\delta_1}\right)^{C({\theta}){\delta_1^{-\frac{1}{\theta}}}},
\end{equation}
where
\begin{eqnarray*}
I&=&\sum_{j}\sum_{*}\sum_{**}|k_jK_j'-k_j'K_j|
e^{2\rho(|j|^{\theta}+(\nu_1^*)^{\theta}-(n_1^*)^{\theta}-(N_1^*)^{\theta})}\\
\nonumber&&\times e^{-(2-2^{\theta})\delta_1\sum_{i\geq 3}|n_i|^{\theta}}
e^{-(2-2^{\theta})\delta_2\sum_{i\geq 3}|N_i|^{\theta}}.
\end{eqnarray*}
To this end, we first note some simple facts:

$\textbf{1}.$ If $j\notin\ \mbox{supp}\ (k+k') \bigcup\ \mbox{supp}\ (K+K')$, then
\begin{equation*}
\frac{\partial\mathcal{M}_{akk'}}{\partial q_j}
\frac{\partial\mathcal{M}_{AKK'}}{\partial \bar{q}_j}-\frac{\partial\mathcal{M}_{akk'}}{\partial \bar{q}_j}
\frac{\partial\mathcal{M}_{AKK'}}{\partial {q}_j}=0.
\end{equation*}
Hence we always assume $j\in\ \mbox{supp}\ (k+k') \bigcup\ \mbox{supp}\ (K+K')$. Therefore one has
\begin{equation*}
|j|\leq \min\{n_1^*,N^*_1\}.
\end{equation*}

$\textbf{2}.$ The following inequality always holds
\begin{equation}\label{041801}
\nu_1^*\leq \max\{n_1^*,N_1^*\},
\end{equation}
and then one has
\begin{equation*}
|j|^{\theta}+(\nu_1^*)^{\theta}-(n_1^*)^{\theta}-(N_1^*)^{\theta}\leq 0.
\end{equation*}

$\textbf{3}.$ It is easy to see
\begin{eqnarray}
\nonumber\sum_{i\geq 1}(n_i^*)^{\theta}
\nonumber&=&\sum_n(2a_n+k_n+k_n')|n|^{\theta}\\
\nonumber&\geq&\sum_n(2a_n+k_n+k_n')\\
\label{042604}&\geq&\sum_n(k_n+k_n')
\end{eqnarray}
and
\begin{eqnarray}
\nonumber\sum_{i\geq 3}(N_i^*)^{\theta}
\nonumber&\geq&\sum_n(2A_n+K_n+K_n')-2\\
\nonumber&\geq&\frac12\sum_n(2A_n+K_n+K_n')\\
\label{042605}&\geq&\frac12\sum_{n}(K_n+K_n').
\end{eqnarray}
Based on (\ref{042604}) and (\ref{042605}), we obtain
\begin{eqnarray}
\sum_{n}(k_n+k_n')(K_n+K_n')\nonumber
&\leq&\nonumber \left(\sup_{n}(K_n+K_n')\right)\left(\sum_{n}(k_n+k_n')\right)\\
&\leq& 2\left(\sum_{i\geq 1}(n_i^*)^{\theta}\right)\left(\sum_{i\geq 3}(N_i^*)^{\theta}\label{042606}\right).
\end{eqnarray}

Now we will prove the inequality (\ref{041802}) holds:

 $\textbf{Case. 1.} \ \nu_1^*\leq N_1^*$.

$\textbf{Case. 1.1.}\ |j|\leq n_3^*.$

Then one has
\begin{eqnarray}
\nonumber e^{2\rho(|j|^{\theta}-(n_1^*)^{\theta})}
&\leq&\nonumber e^{2\rho((n_3^*)^{\theta}-(n_1^*)^{\theta})}\\
&\leq& e^{(2-2^{\theta})\delta_1((n_3^*)^{\theta}-(n_1^*)^{\theta})}\label{051001},
\end{eqnarray}
where using $0<\delta_1\ll1$ depending on $\theta$ and $\rho$.
Hence
\begin{eqnarray}
\nonumber &&e^{2\rho(|j|^{\theta}+(\nu_1^*)^{\theta}-(n_1^*)^{\theta}-(N_1^*)^{\theta})}e^{-(2-2^{\theta})\delta_1\sum_{i\geq 3}(n_i^*)^{\theta}}\\
&\leq&\nonumber e^{(2-2^{\theta})\delta_1((n_3^*)^{\theta}-(n_1^*)^{\theta})}
e^{-(2-2^{\theta})\delta_1\sum_{i\geq 3}(n_i^*)^{\theta}}\\
&&\nonumber \mbox{(in view of $\nu_1^*\leq N_1^*$ and (\ref{051001}))}\\
&=&\nonumber e^{-(2-2^{\theta})\delta_1((n_1^*)^{\theta}+\sum_{i\geq 4}(n_i^*)^{\theta})}\\
&\leq& e^{-\frac{(2-2^{\theta})\delta_1}3\sum_{i\geq 1}(n_i^*)^{\theta}}\label{042603}.
\end{eqnarray}

\begin{remark}\label{042607}
Note that if $j,a,k,k'$ are specified, and then $A,K,K'$ are uniquely determined.
\end{remark}

In view of (\ref{042603}), we have
\begin{eqnarray*}
I&\leq&\sum_{j}\sum_{*}\sum_{**}(k_j+k_j')(K_j+K_j')
e^{-\frac{(2-2^{\theta})\delta_1}3\sum_{i\geq 1}(n_i^*)^{\theta}}e^{-(2-2^{\theta})\delta_2\sum_{i\geq 3}(N_i^*)^{\theta}}\\
&=&\sum_{a,k,k'}\sum_{j}(k_j+k_j')(K_j+K_j')
e^{-\frac{(2-2^{\theta})\delta_1}3\sum_{i\geq 1}(n_i^*)^{\theta}}e^{-(2-2^{\theta})\delta_2\sum_{i\geq 3}(N_i^*)^{\theta}}\\
&&\mbox{(in view of Remark \ref{042607}, one has $\sum_{a,k,k'}\sum_{j}=\sum_{*}\sum_{**}\sum_{j}$)}\\
&\leq&\sum_{a,k,k'}2\left(\sum_{i\geq 1}(n_i^*)^{\theta}\right)\left(\sum_{i\geq 3}(N_i^*)^{\theta}\label{042606}\right)e^{-\frac{(2-2^{\theta})\delta_1}3\sum_{i\geq 1}(n_i^*)^{\theta}}e^{-(2-2^{\theta})\delta_2\sum_{i\geq 3}(N_i^*)^{\theta}}\\
&&\mbox{(in view of the inequality (\ref{042606}))}\\
&=&2\sum_{a,k,k'}\left(\sum_{i\geq 1}(n_i^*)^{\theta}
e^{-\frac{(2-2^{\theta})\delta_1}3\sum_{i\geq 1}(n_i^*)^{\theta}}\right)\left(\sum_{i\geq 3}(N_i^*)^{\theta}e^{-(2-2^{\theta})\delta_2\sum_{i\geq 3}|N_i|^{\theta}}\right)\\
&\leq&\frac{2}{(2-2^{\theta})\delta_2}\sum_{a,k,k'}\sum_{i\geq 1}(n_i^*)^{\theta}
e^{-\frac{(2-2^{\theta})\delta_1}3\sum_{i\geq 1}(n_i^*)^{\theta}}\qquad \mbox{(in view of (\ref{042805}))}\\&\leq&\frac{24}{(2-2^{\theta})^2\delta_1\delta_2}\sum_{a,k,k'}
e^{-\frac{(2-2^{\theta})\delta_1}4\sum_{i\geq 1}(n_i^*)^{\theta}}\qquad \mbox{(in view of (\ref{042805}) again)}\\
&=&\frac{24}{(2-2^{\theta})^2\delta_1\delta_2}\sum_{a,k,k'}
e^{-\frac{(2-2^{\theta})\delta_1}4\sum_{n}(2a_n+k_n+k'_n)|n|^{\theta}}\\
&\leq&\frac{24}{(2-2^{\theta})^2\delta_1\delta_2}\left(\sum_{a}
e^{-\frac{(2-2^{\theta})\delta_1}4\sum_{n}2a_n|n|^{\theta}}\right)
\left(\sum_{k}e^{-\frac{(2-2^{\theta})\delta_1}4\sum_{n}k_n|n|^{\theta}}\right)^2
\\
&\leq&\frac{24}{(2-2^{\theta})^2\delta_1\delta_2}\prod_{n\in\mathbb{Z}}\left({1-e^{-\frac{(2-2^{\theta})\delta_1}2|n|^{\theta}}}\right)^{-1}
\left({1-e^{-\frac{(2-2^{\theta})\delta_1}4|n|^{\theta}}}\right)^{-2}      \\
&&\nonumber\mbox{(which is based on Lemma \ref{a3})}\\
&\leq&\frac{C_1(\theta)}{\delta_1\delta_2}\left(\frac{1}{\delta_1}\right)^{C_2{(\theta)}{\delta_1^{-\frac{1}{\theta}}}}\ \ \mbox{(in view of (\ref{0418010}))}\\
&\leq&\frac{1}{\delta_2}\left(\frac{1}{\delta_1}\right)^{C{(\theta)}{\delta_1^{-\frac{1}{\theta}}}},
\end{eqnarray*}
where the last inequality is based on $0<\delta_1\ll1$ and $C(\theta), C_1(\theta), C_2(\theta)$ are positive constants depending on $\theta$ only.

$\textbf{Case. 1.2.}\ j\in\{n_1,n_2\},\ |n_1|=n_1^*,\ |n_2|=n_2^*.$

If $2a_j+k_j+k'_j>2$, then $|j|\leq n_3^*$, we are in $\textbf{Case. 1.1.}$. Hence in what follows, we always assume
\begin{equation*}
2a_j+k_j+k'_j\leq2,
\end{equation*}
which implies
\begin{equation}\label{2.24}
k_j+k_j'\leq 2.
\end{equation}
From (\ref{2.24}) and in view of $j\in\{n_1,n_2\}$, it follows that
\begin{eqnarray*}
I\nonumber
\nonumber&\leq&2\sum_{a,k,k'}(K_{n_1}+K'_{n_1}+K_{n_2}+K'_{n_2}) \\
\label{2.25}&&\times e^{-(2-2^{\theta})\delta_1\sum_{i\geq 3}(n_i^*)^{\theta}{-(2-2^{\theta})\delta_2\sum_{i\geq 3}(N_i^*)^{\theta}}}.
\end{eqnarray*}
In view of (\ref{052802}) and (\ref{042605}), we have
\begin{eqnarray}
\sum_{n}(2\alpha_n+\kappa_n+\kappa'_n)
\leq\label{042701}2\sum_{i\geq3}(n_{i}^{*})^{\theta}+2\sum_{i\geq3}(N_{j}^{*})^{\theta}.
\end{eqnarray}
Moreover, note that  $\forall j$,
\begin{equation*}
 K_j+K'_j\leq \kappa_j+\kappa'_j-k_j-k'_j+2\leq\kappa_j+\kappa'_j+2.
\end{equation*}
Hence,
\begin{eqnarray}
\nonumber I\nonumber&\leq& 2\sum_{a,k,k'}(\kappa_{n_1}+\kappa'_{n_1}+\kappa_{n_2}+\kappa'_{n_2}+4) \\
\nonumber&&\times e^{-(2-2^{\theta})\delta_1\sum_{i\geq 3}(n_i^*)^{\theta}{-(2-2^{\theta})\delta_2\sum_{i\geq 3}(N_i^*)^{\theta}}}\\
\nonumber &\leq& 2\sum_{a,k,k'}(\kappa_{n_1}+\kappa'_{n_1}+\kappa_{n_2}+\kappa'_{n_2}+4)\\
\nonumber&&\times e^{-\frac12(2-2^{\theta})\delta_1\sum_{i\geq 3}(n_i^*)^{\theta}}e^{{-\frac12(2-2^{\theta})(\delta_2\sum_{i\geq 3}(N_i^*)^{\theta}+\delta_1\sum_{i\geq 3}(n_i^*)^{\theta})}}\\
\label{2.26} &\leq& 2\sum_{a,k,k'}(\kappa_{n_1}+\kappa'_{n_1}+\kappa_{n_2}+\kappa'_{n_2}+4) \\
\nonumber&&\times e^{-\frac12(2-2^{\theta})\delta_1\sum_{i\geq 3}(n_i^*)^{\theta}}e^{-\frac14(\delta_1\wedge\delta_2)(2-2^\theta)\sum_{n}(2\alpha_n+\kappa_n+\kappa'_n)},
\end{eqnarray}
where the last inequality is based on (\ref{042701}) and $\delta_1\wedge\delta_2=\min\{\delta_1,\delta_2\}$.

\begin{remark}\label{042703}Obviously, $\{n_1,n_2\}\cap \mathrm{supp}\ \mathcal{M}_{\alpha\kappa\kappa'}\neq \emptyset$, and if $\{n_i\}_{i\geq 3}$ and $n_1$ (resp. $n_2$) is specified, then $n_2$ (resp. $n_1$) is determined uniquely. Thus $n_1,n_2$ range in a set of cardinality no more than \begin{equation}\label{042702}\#\mathrm{supp} \ \mathcal{M}_{\alpha\kappa\kappa'}\leq\sum_{n}(2\alpha_n+\kappa_n+\kappa'_n).
\end{equation}
\end{remark}
Also, if $\{n_i\}_{i\geq 1}$ is given, then $\{2a_n+k_n+k'_n\}_{n\in\mathbb{Z}}$ is specified, and hence $(a,k,k')$ is specified up to a factor of
$$\prod_{n}(1+l_n^2),$$
where
$$l_n=\#\{j:n_j=n\}.$$
Following the inequality (\ref{2.26}), we thus obtain
\begin{eqnarray}
\nonumber I&\leq& 2\sum_{\{n_i\}_{i\geq1}}\prod_{m}(1+l_m^2)(\kappa_{n_1}+\kappa'_{n_1}+\kappa_{n_2}+\kappa'_{n_2}+4) \\
\nonumber&&\times e^{-\frac12(2-2^{\theta})\delta_1\sum_{i\geq 3}(n_i^*)^{\theta}}e^{-\frac14(\delta_1\wedge\delta_2)(2-2^\theta)\sum_{n}(2\alpha_n+\kappa_n+\kappa'_n)}\\
\nonumber&=& 2\sum_{\{n_i\}_{i\geq3}}\prod_{m}(1+l_m^2)\left(\sum_{n_1,n_2}(\kappa_{n_1}+\kappa'_{n_1}+\kappa_{n_2}+\kappa'_{n_2}+4)\right) \\
\nonumber&&\times e^{-\frac12(2-2^{\theta})\delta_1\sum_{i\geq 3}(n_i^*)^{\theta}}e^{-\frac14(\delta_1\wedge\delta_2)(2-2^\theta)\sum_{n}(2\alpha_n+\kappa_n+\kappa'_n)}\\
\nonumber& \leq & 2\sum_{\{n_i\}_{i\geq3}}\prod_{m}(1+l_m^2)\left(\sum_{n}(\kappa_{n}+\kappa'_{n})+4\#\mathrm{supp}\ \mathcal{M}_{\alpha\kappa\kappa'}\right) \\
\nonumber&&\times e^{-\frac12(2-2^{\theta})\delta_1\sum_{i\geq 3}(n_i^*)^{\theta}
}e^{-\frac14(\delta_1\wedge\delta_2)(2-2^\theta)\sum_{n}(2\alpha_n+\kappa_n+\kappa'_n)}\\
&&\nonumber\mbox{(the inequality is based on Remark \ref{042703})}\\
\nonumber&\leq& 10\sum_{\{n_i\}_{i\geq3}}\prod_{m}(1+l_m^2)e^{-\frac12(2-2^{\theta})\delta_1\sum_{i\geq 3}(n_i^*)^{\theta}} \\
\nonumber&&\times\left(\sum_{n}(2\alpha_n+\kappa_{n}+\kappa'_{n})\right)e^{
-\frac14(\delta_1\wedge\delta_2)(2-2^\theta)\sum_{n}(2\alpha_n+\kappa_n+\kappa'_n)}\\
&&\mbox{ (based on (\ref{042702}))}\nonumber\\
&\leq & \nonumber\frac{C_3({\theta})}{\delta_1\wedge\delta_2}\sum_{\{n_i\}_{i\geq3}}\prod_{|m|\leq|n_1|}(1+l_m^2)e^{-\frac12(2-2^{\theta})\delta_1\sum_{i\geq 3}(n_i^*)^{\theta}}\quad \\
&&\nonumber\mbox{(in view of (\ref{042805}))}\\
\nonumber&=&\frac{C_3({\theta})}{\delta_1\wedge\delta_2}\sum_{\{l_m\}_{|m|\leq|n_3|}}\left(\prod_{|m|\leq|n_1|}(1+l_m^2)e^{-\frac16(2-2^{\theta})\delta_1\sum_{|m|\leq|n_3|}l_m|m|^{\theta}}\right)
\nonumber\\&&\times \nonumber e^{-\frac13(2-2^{\theta})\delta_1\sum_{|m|\leq|n_3|}l_m|m|^{\theta}}\\
\nonumber&\leq&\frac{C_3({\theta})}{\delta_1\wedge\delta_2}\sup_{\{l_m\}_{|m|\leq|n_3|}}\left(\prod_{|m|\leq|n_1|}(1+l_m^2)e^{-\frac16(2-2^{\theta})\delta_1\sum_{|m|\leq|n_3|}l_m|m|^{\theta}}\right)
\\
\nonumber&&\times\sum_{\{l_m\}_{|m|\leq|n_3|}}e^{-\frac13(2-2^{\theta})\delta_1\sum_{|m|\leq|n_3|}l_m|m|^{\theta}}\\
\nonumber&\leq&\frac{C_3({\theta})}{\delta_1\wedge\delta_2}
\left(\frac{1}{\delta_1}\right)^{C_4({\theta}){\delta_1^{-\frac{1}{\theta}}}}\sum_{\{l_m\}_{|m|\leq|n_3|}}e^{-\frac13(2-2^{\theta})\delta_1\sum_{|m|\leq|n_3|}l_m|m|^{\theta}}\\
\nonumber&&\mbox{(in view of (\ref{201606031}))}\\
\nonumber& \leq & \frac{C_3({\theta})}{\delta_1\wedge\delta_2}\left(\frac{1}{\delta_1}\right)^{C_4({\theta}){\delta_1^{-\frac{1}{\theta}}}}\prod_{m\in\mathbb{Z}}\frac{1}{1-e^{-\frac13(2-2^{\theta})\delta_1 |m|^{\theta}}}\quad\\
&&\nonumber \mbox{(in view of (\ref{041809}))}\\
\nonumber&\leq&\frac{1}{\delta_2}\left(\frac{1}{\delta_1}\right)^{C({\theta}){\delta_1^{-\frac{1}{\theta}}}}\ \mbox{(in view of (\ref{0418010}) and $0<\delta_1\ll1$)},
\end{eqnarray}
where $C(\theta)$ is a positive constant depending on $\theta$ only.

$\textbf{Case. 2.}\ \nu_1^*>N_1^*.$

In view of (\ref{041801}), one has $n_1^*=\nu_1^*$. Hence,  $n_2$ is determined by $n_1$ and $\{n_i\}_{i\geq 3}$. Similar as Case 1.2, we have
\begin{eqnarray*}
I
&\leq& \frac{C'(\theta)}{\delta_2}\sum_{\{n_i\}_{i\geq3}}e^{-\frac12(2-2^{\theta})\delta_1\sum_{i\geq 3}(n_i^*)^{\theta}}\prod_{m}(1+l_m^2)\\
&\leq&\frac{1}{\delta_2}\left(\frac{1}{\delta_1}\right)^{C({\theta}){\delta_1^{-\frac{1}{\theta}}}},
\end{eqnarray*}
where $C'(\theta)$ is some positive constant depending on $\theta$ only.

\end{proof}

Next, we will estimate the symplectic transformation $\Phi_F$ induced by the Hamiltonian function $F$. Actually, we have
\begin{lemma}\label{E1}
Let $\theta\in(0,1),\rho>0$ and $0<\delta_1,\delta_2\ll1$ (depending on $\theta,\rho$). Assume further \begin{equation}\label{042801}\frac{1}{\delta_2}\left(\frac{1}{\delta_1}\right)^{C({\theta}){\delta_1^{-\frac{1}{\theta}}}}||F||_{\rho-\delta_1} \ll 1,
\end{equation}
where $C(\theta)$ is the constant given in (\ref{042704}) in Lemma \ref{010}.
Then for any Hamiltonian function $H$, we get
\begin{equation}
||H\circ\Phi_F||_\rho
\leq\left(1+\frac{1}{\delta_2}\left(\frac{1}{\delta_1}\right)^{C_1({\theta}){\delta_1^{-\frac{1}{\theta}}}}||F||_{\rho-\delta_1}\right)
||H||_{\rho-\delta_2},
\end{equation}
where $C_1(\theta)$ is a positive constant depending only on $\theta$.
\end{lemma}

\begin{proof}
 Firstly,  we expand $H\circ\Phi_F$ into the Taylor series
 \begin{equation}\label{3.2}
 H\circ\Phi_F=\sum_{n\geq 0}\frac{1}{n!}H^{(n)},
 \end{equation}
where $H^{(n)}=\{H^{(n-1)},F\}$ and $H^{(0)}=H$.

We will estimate $||H^{(n)}||_\rho$ by using Lemma \ref{010} again and again:
\begin{eqnarray}
\nonumber||H^{(n)}||_\rho
\nonumber&=&||\{H^{(n-1)},F\}||_\rho\\
\nonumber&\leq&\left(\left(\frac{1}{\delta_1}\right)^{C({\theta}){\delta_1^{-\frac{1}{\theta}}}}
||F||_{\rho-\delta_1}\right)\left(\frac{n}{\delta_2}\right)||H^{(n-1)}||_{\rho-\frac{\delta_2}{n}}\\
\nonumber&\leq&\left(\left(\frac{1}{\delta_1}\right)^{C({\theta}){\delta_1^{-\frac{1}{\theta}}}}
||F||_{\rho-\delta_1}\right)^2\left(\frac{n}{\delta_2}\right)^2||H^{(n-2)}||_{\rho-\frac{2\delta_2}{n}}\\
&&\dots\nonumber\\
\label{3.3}&\leq&\left(\left(\frac{1}{\delta_1}\right)^{C({\theta}){\delta_1^{-\frac{1}{\theta}}}}
||F||_{\rho-\delta_1}\right)^n\left(\frac{n}{\delta_2}\right)^n||H||_{\rho-\delta_2}.
\end{eqnarray}
Hence in view of (\ref{3.2}), one has
\begin{eqnarray*}
||H\circ\Phi_F||_\rho
&\leq&\sum_{n\geq 0}\frac{1}{n!}\left(\left(\frac{1}{\delta_1}\right)^{C({\theta}){\delta_1^{-\frac{1}{\theta}}}}
||F||_{\rho-\delta_1}\right)^{n}\left(\frac{n}{\delta_2}\right)^n||H||_{\rho-\delta_2}\\
&=&\sum_{n\geq 0}\frac{n^n}{n!}\left(\frac1{\delta_2}\left(\frac{1}{\delta_1}\right)^{C({\theta}){\delta_1^{-\frac{1}{\theta}}}}
||F||_{\rho-\delta_1}\right)^{n}||H||_{\rho-\delta_2}\\
&\leq&\sum_{n\geq 0}\left(\frac e{\delta_2}\left(\frac{1}{\delta_1}\right)^{C({\theta}){\delta_1^{-\frac{1}{\theta}}}}
||F||_{\rho-\delta_1}\right)^{n}||H||_{\rho-\delta_2}\\
&&(\mbox{in view of $n^n<n!e^n$)}\\
&\leq & \left(1+\frac{1}{\delta_2}\left(\frac{1}{\delta_1}\right)^{C({\theta}){\delta_1^{-\frac{1}{\theta}}}}||F||_{\rho-\delta_1}\right)
||H||_{\rho-\delta_2}\\
 &&\mbox{(in view of (\ref{042801}) and $0<\delta_1\ll1$)},
\end{eqnarray*}
where $C_1(\theta)$ is a positive constant depending on $\theta$ only.
\end{proof}

\subsection{The new perturbation $R_+$ and the new normal form $N_+$}

Firstly, for $i=0,1$, one has
\begin{eqnarray}
\nonumber||{F}_i||_{\rho+\delta}
&\leq&\nonumber\frac{C_1(\theta)}{{\delta^2}}||F_i||_{\rho+\frac{\delta}2}^+\qquad\qquad\qquad\qquad\mbox{(in view of  (\ref{N7}))}\\
&\leq&\nonumber \frac{C_1(\theta)}{{\delta^2}}\cdot \frac{e^3}{\gamma}\cdot e^{{C_2({\theta)}}{\delta^{-\frac5{\theta}}}}||R_i||_{\rho}^+\qquad\mbox{(in view of  (\ref{S7}))}  \\
&\leq&\label{N9}\frac{1}{\gamma\delta^3}e^{{\delta^{-\frac6{\theta}}}}||R_i||_{\rho}^+\qquad \ \qquad\qquad\quad \mbox{(by assuming $\delta\ll 1$)}.
\end{eqnarray}
Recall the new term $R_+$ is given by (\ref{051403}), i.e. $R_+=R_2+\mathcal{R}$, where
\begin{equation*}
\mathcal{R}=\int_{0}^1\{(1-t)\{N,F\}+R_0+R_1,F\}\circ X_F^{t}dt+\int_0^1\{R_2,F\}\circ X_F^tdt.
\end{equation*}
Write
\begin{equation}
R_+=R_{0+}+R_{1+}+R_{2+},
\end{equation}
%Assume\begin{equation}\label{N10}
%\left(\frac{1}{\delta}\right)^{{\delta^{-\frac{2}{\theta}}}}\cdot\frac{1}{\gamma\delta^3}e^{{\delta^{-\frac6{\theta}}}}(||R_0||_{\rho}^++||R_1||_{\rho}^+)\ll1,
%\end{equation}
%which implies,
%\begin{eqnarray*}
%&&\left(\frac{1}{\delta}\right)^{{\delta^{-\frac{2}{\theta}}}}||{F}_i||_{\rho+\delta}\\
%&\leq&\left(\frac{1}{\delta}\right)^{{\delta^{-\frac{2}{\theta}}}}\frac{1}{\gamma\delta^3}e^{{\delta^{-\frac{6}\theta}}}||F_i||_{\rho}^+\qquad %\mbox{(in view of (\ref{N10}))}\\
%&\ll&1,
%\end{eqnarray*}
%Hence the condition (\ref{042801}) in Lemma \ref{E1} is fulfilled with %$\delta_1=\delta_2=\delta$ and $\rho$ replaced by $\rho+2\delta$.
and write $\mathcal{R}$ as
\begin{eqnarray}
\nonumber\mathcal{R}
&=&\label{051510}\int_0^1\{R_0,F\}\circ X_F^tdt\\
&&\label{051511}+\int_0^1\{R_1,F\}\circ X_F^tdt\\
&&\label{051513}+\int_0^1\{R_2,F\}\circ X_F^tdt\\
&&\label{051520}+\int_0^1(1-t)\{\{N,F\},F\}\circ X_F^tdt.
\end{eqnarray}
Firstly note that the term $(\ref{051510})$ contributes to $R_{0+},R_{1+},R_{2+}$ and we get
\begin{eqnarray}
\nonumber\left|\left|\int_0^1\{R_0,F\}\circ X_F^tdt\right|\right|_{\rho+2\delta}
\nonumber&=&\left|\left|\sum_{n\geq1}\frac{1}{n!}\underbrace{\left\{\cdots\left\{R_0,{F}\right\},\cdots,{F}\right\}}_{n-\mbox{fold}}\right|\right|_{\rho+2\delta}\\
\nonumber&\leq& \frac{1}\delta\cdot\left(\frac{1}{\delta}\right)^{C_1(\theta){\delta^{-\frac{1}{\theta}}}}||{F}||_{\rho+\delta}||R_{0}||_{\rho+\delta}\ \\
&&\nonumber \mbox{(following the proof of Lemma \ref{E1})}\\
\nonumber&\leq&\frac{1}{\delta}\left(\frac{1}{\delta}\right)^{C_1(\theta){\delta^{-\frac{1}{\theta}}}}\cdot
\frac{1}{\gamma\delta^3}e^{{\delta^{-\frac6{\theta}}}}\cdot\frac{C_2(\theta)}{\delta^2}(||R_0||_{\rho}^++||R_1||_{\rho}^+)||R_0||_{\rho}^+\ \\&&\nonumber \mbox{(by (\ref{N7}) and (\ref{N9}))}\\
&\leq&\label{051203} \frac1{\gamma}\cdot e^{{\delta^{-\frac{8}{\theta}}}}||R_0||_{\rho}^+(||R_0||_{\rho}^++||R_1||_{\rho}^+),
\end{eqnarray}
and consequently

\begin{eqnarray}
\nonumber\left|\left|\int_0^1\{R_0,F\}\circ X_F^tdt\right|\right|_{\rho+3\delta}^{+}
\nonumber&\leq& \left(\frac{1}{\delta}\right)^{C(\theta){\delta^{-\frac{1}{\theta}}}}\left|\left|\int_0^1\{R_0,F\}\circ X_F^tdt\right|\right|_{\rho+2\delta}\ \ \mbox{(from (\ref{N6}))}\\
\nonumber&\leq&\left(\frac{1}{\delta}\right)^{C(\theta){\delta^{-\frac{1}{\theta}}}} \cdot\frac{1}{\gamma}\cdot e^{{\delta^{-\frac{8}{\theta}}}}||R_0||_{\rho}^+(||R_0||_{\rho}^++||R_1||_{\rho}^+)\qquad{\mbox{(from (\ref{051203}))}}\\
\label{N11}&\leq& \frac1{\gamma}\cdot e^{{\delta^{-\frac{10}{\theta}}}}||R_0||_{\rho}^+(||R_0||_{\rho}^++||R_1||_{\rho}^+).
\end{eqnarray}
Secondly, we consider the term (\ref{051511}) and write
\begin{eqnarray}
(\ref{051511})&=&\nonumber\sum_{n\geq1}\frac{1}{n!}\underbrace{\{\cdots\{R_1,{F}\},{F},\cdots,{F}\}}_{n-\mbox{fold}}\\
&=&\label{051505}\sum_{n\geq1}\frac{1}{n!}\underbrace{\{\cdots\{R_1,{F}_0\},{F},\cdots,{F}\}}_{(n-1)-\mbox{fold}}\\
&&\label{051512}+\{R_1,F_1\}\\
&&\label{051506}+\sum_{n\geq2}\frac{1}{n!}\underbrace{\{\cdots\{R_1,F_1\},{F},\cdots,{F}\}}_{(n-1)-\mbox{fold}}.
\end{eqnarray}
Note that (\ref{051505}) contributes to $R_{0+},R_{1+},R_{2+}$, (\ref{051512}) contributes to $R_{1+},R_{2+}$ and (\ref{051506}) contributes to $R_{0+},R_{1+},R_{2+}$.

Moreover, following the proof of (\ref{N11}), one has
\begin{eqnarray}
\label{N12}\left|\left|(\ref{051505})\right|\right|_{\rho+3\delta}^{+}
&\leq& \frac1{\gamma}\cdot e^{{\delta^{-\frac{10}{\theta}}}}||R_1||_{\rho}^+||R_0||_{\rho}^+(||R_0||_{\rho}^++||R_1||_{\rho}^+),\\
\left|\left|(\ref{051512})\right|\right|_{\rho+3\delta}^{+}&\leq& \frac1{\gamma}\cdot e^{{\delta^{-\frac{10}{\theta}}}}||R_1||_{\rho}^+||R_1||_{\rho}^+,\\
\label{N14}\left|\left|(\ref{051506})\right|\right|_{\rho+3\delta}^{+}
&\leq& \frac1{\gamma}\cdot e^{{\delta^{-\frac{10}{\theta}}}}||R_1||_{\rho}^+||R_1||_{\rho}^+(||R_0||_{\rho}^++||R_1||_{\rho}^+).
\end{eqnarray}
Thirdly, we consider the term (\ref{051513}) and write
\begin{eqnarray}
(\ref{051513})&=&\nonumber\sum_{n\geq1}\frac{1}{n!}\underbrace{\{\cdots\{R_2,{F}\},{F},\cdots,{F}\}}_{n-\mbox{fold}}\\
&=&\label{051514}\{R_2,F_0\}\\
&&+\label{051515}\{R_2,F_1\}\\
&&+\label{051516}\sum_{n\geq2}\frac{1}{n!}\underbrace{\{\cdots\{R_2,{F}_0\},{F},\cdots,{F}\}}_{(n-1)-\mbox{fold}}\\
&&\label{051517}+\{\{R_1,F_1\},F\}\\
&&\label{051518}+\sum_{n\geq3}\frac{1}{n!}\underbrace{\{\cdots\{R_2,F_1\},{F},\cdots,{F}\}}_{(n-1)-\mbox{fold}}.
\end{eqnarray}
Note that (\ref{051514}) contributes to $R_{1+},R_{2+}$, (\ref{051515}) contributes to $R_{2+}$, (\ref{051516}) contributes to $R_{0+},R_{1+},R_{2+}$, (\ref{051517}) contributes to $R_{1+},R_{2+}$ and (\ref{051518}) contributes to $R_{0+},R_{1+},R_{2+}$.

Similarly, one has
\begin{eqnarray}
\label{N15}||(\ref{051514})||_{\rho+3\delta}^{+}&\leq& \frac1{\gamma}\cdot e^{{\delta^{-\frac{10}{\theta}}}}||R_0||_{\rho}^+||R_2||_{\rho}^+,\\
\label{N17}||(\ref{051515})||_{\rho+3\delta}^{+}
&\leq& \frac1{\gamma}\cdot e^{{\delta^{-\frac{10}{\theta}}}}||R_2||_{\rho}^+||R_1||_{\rho}^+,\\
\label{N16}\left|\left|(\ref{051516})\right|\right|_{\rho+3\delta}^{+}
&\leq& \frac1{\gamma}\cdot e^{{\delta^{-\frac{10}{\theta}}}}||R_0||_{\rho}^+||R_2||_{\rho}^+(||R_0||_{\rho}^++||R_1||_{\rho}^+),\\
\label{N18}\left|\left|(\ref{051517})\right|\right|_{\rho+3\delta}^{+}
&\leq& \frac{1}{\gamma}\cdot e^{{\delta^{-\frac{10}{\theta}}}}||R_2||_{\rho}^+||R_1||_{\rho}^+(||R_0||_{\rho}^++||R_1||_{\rho}^+),\\
\label{N19}\left|\left|(\ref{051518})\right|\right|_{\rho+3\delta}^{+}
&\leq& \frac1{\gamma}\cdot e^{{\delta^{-\frac{10}{\theta}}}}||R_2||_{\rho}^+||R_1||_{\rho}^+(||R_0||_{\rho}^++||R_1||_{\rho}^+)^2.
\end{eqnarray}
Finally, we consider the term (\ref{051520}) and
write
\begin{eqnarray}
\nonumber(\ref{051520})&=&\sum_{n\geq2}\frac{1}{n!}\underbrace{\{\cdots\{N,{F}\},{F},\cdots,{F}\}}_{n-\mbox{fold}}\\
&=&\nonumber\sum_{n\geq2}\frac{1}{n!}\underbrace{\{\cdots\{-R_0-R_1+[R_0]+[R_1]\},{F},\cdots,{F}\}}_{(n-1)-\mbox{fold}}\qquad \\
&&\nonumber\mbox{(in view of (\ref{4.27}))}\\
&=&\label{051530}\sum_{n\geq2}\frac{1}{n!}\underbrace{\{\cdots\{-R_0+[R_0],{F}\},{F},\cdots,{F}\}}_{(n-1)-\mbox{fold}}\\
&&+\label{051531}\sum_{n\geq2}\frac{1}{n!}\underbrace{\{\cdots\{-R_1+[R_1],{F}_0\},{F},\cdots,{F}\}}_{(n-2)-\mbox{fold}}\\
&&\label{051532}+\{-R_1+[R_1],F_1\}\\
&&\label{051533}+\sum_{n\geq3}\frac{1}{n!}\underbrace{\{\cdots\{-R_1+[R_1],F_1\},{F},\cdots,{F}\}}_{(n-1)-\mbox{fold}}.
\end{eqnarray}

Note that (\ref{051530}) contributes to $R_{0+},R_{1+},R_{2+}$, (\ref{051531}) contributes to $R_{0+},R_{1+},R_{2+}$, (\ref{051532}) contributes to $R_{1+},R_{2+}$ and (\ref{051533}) contributes to $R_{0+},R_{1+},R_{2+}$.

Moreover, one has
\begin{eqnarray}
\left|\left|(\ref{051530})\right|\right|_{\rho+3\delta}^{+}&\leq&\frac1{\gamma}\cdot e^{{\delta^{-\frac{10}{\theta}}}}||R_0||_{\rho}^+(||R_0||_{\rho}^++||R_1||_{\rho}^+),\\
\label{N12'}\left|\left|(\ref{051531})\right|\right|_{\rho+3\delta}^{+}
&\leq& \frac1{\gamma}\cdot e^{{\delta^{-\frac{10}{\theta}}}}||R_1||_{\rho}^+||R_0||_{\rho}^+(||R_0||_{\rho}^++||R_1||_{\rho}^+),\\
\left|\left|(\ref{051532})\right|\right|_{\rho+3\delta}^{+}&\leq& \frac1{\gamma}\cdot e^{{\delta^{-\frac{10}{\theta}}}}||R_1||_{\rho}^+||R_1||_{\rho}^+,\\
\label{N14}\left|\left|(\ref{051533})\right|\right|_{\rho+3\delta}^{+}
&\leq& \frac1{\gamma}\cdot e^{{\delta^{-\frac{10}{\theta}}}}||R_1||_{\rho}^+||R_1||_{\rho}^+(||R_0||_{\rho}^++||R_1||_{\rho}^+).
\end{eqnarray}
Consequently
\begin{eqnarray}
\label{N20}||R_{0+}||_{\rho+3\delta}^{+}&\leq& \frac1{\gamma}\cdot e^{{\delta^{-\frac{10}{\theta}}}}(||R_0||_{\rho}^++||R_1||_{\rho}^+)(||R_0||_{\rho}^++{||R_1||_{\rho}^+}^2),\\
\label{N21}||R_{1+}||_{\rho+3\delta}^{+}&\leq& \frac1{\gamma}\cdot e^{{\delta^{-\frac{10}{\theta}}}}(||R_0||_{\rho}^++{||R_1||_{\rho}^+}^2),\\
\label{N22}||R_{2+}||_{\rho+3\delta}^{+}&\leq& ||R_2||_{\rho}^++\frac1{\gamma}\cdot e^{{\delta^{-\frac{10}{\theta}}}}(||R_0||_{\rho}^++||R_1||_{\rho}^+).
\end{eqnarray}

The new normal form $N_+$ is given in (\ref{051402}). Note that $[R_0]$ (in view of (\ref{051501})) is a constant which does not affect the Hamiltonian vector field. Moreover, in view of (\ref{051502}), we denote by
\begin{equation}
\omega_{n+}=n^2+\widetilde V_n+\sum_aB_{a00}^{(n)}\mathcal{M}_{a00},
\end{equation}
where the terms $\sum_aB_{a00}^{(n)}\mathcal{M}_{a00}$ is the so-called frequency shift. The estimate of $|\sum_aB_{a00}^{(n)}\mathcal{M}_{a00}|$ will be given in the next section (see (\ref{M15}) for the details).

\section{Iteration and Convergence}

Now we give the precise
set-up of iteration parameters. Let $s\geq1$ be the $s$-th KAM
step.
 \begin{itemize}
 \item[]$\delta_{s}=\frac{\rho}{s^2}$,

 \item[]$\rho_{s+1}=\rho_{s}+3\delta_s$,

 \item[]$\epsilon_s=\epsilon_{0}^{(\frac{3}{2})^s}$, which dominates the size of
 the perturbation,

 \item[]$\lambda_s=e^{-C(\theta)(\ln{\frac{1}{\epsilon_{s+1}}})^{\frac{4}{\theta+4}}}$,

 \item[]$\eta_{s+1}=\frac{1}{20}\lambda_s\eta_s$,

 \item[]$d_0=0,\,d_{s+1}=d_s+\frac{1}{\pi^2(s+1)^2}$,

 \item[]$D_s=\{(q_n)_{n\in\mathbb{Z}}:\frac{1}{2}+d_s\leq|q_n|e^{r|n|^{\theta}}\leq1-d_s\}$.
 \end{itemize}

Denote the complex cube of size $\lambda>0$:
\begin{equation}\label{M9}
\mathcal{C}_{\lambda}(\widetilde{V})=\{(V_n)_n\in\prod_{n\in\mathbb{Z}}\mathbb{C}:|V_n-\widetilde{V}_n|\leq \lambda\}.
\end{equation}

\begin{lemma}{\label{IL}}
Suppose $H_{s}=N_{s}+R_{s}$ is real analytic on $D_{s}\times\mathcal{C}_{\eta_{s}}(V_{s})$,
where $$N_{s}=\sum_{n\in\mathbb{Z}}(n^2+\widetilde V_{n,s})|q_n|^2$$ is a normal form with coefficients satisfying
\begin{eqnarray}
\label{198}&&\widetilde{V}_{s}(V_{s})=\omega,\\
\label{199}&&\left|\left|\frac{\partial \widetilde{V}_s}{{\partial V}}-I\right|\right|_{l^{\infty}\rightarrow l^{\infty}}<d_s\epsilon_{0}^{\frac{1}{10}},
\end{eqnarray}
and $R_{s}=R_{0,s}+R_{1,s}+R_{2,s}$ satisfying
\begin{eqnarray}
\label{200}&&||R_{0,s}||_{\rho_{s}}^{+}\leq \epsilon_{s},\\
\label{201}&&||R_{1,s}||_{\rho_{s}}^{+}\leq \epsilon_{s}^{0.6},\\
\label{202}&&||R_{2,s}||_{\rho_{s}}^{+}\leq (1+d_s)\epsilon_0.
\end{eqnarray}
Then for all $V\in\mathcal{C}_{\eta_{s}}(V_{s})$ satisfying $\widetilde V_{s}(V)\in\mathcal{C}_{\lambda_s}(\omega)$, there exist real analytic symplectic coordinate transformations
$\Phi_{s+1}:D_{s+1}\rightarrow D_{s}$ satisfying
\begin{eqnarray}
\label{203}&&||\Phi_{s+1}-id||_{r,\infty}\leq \epsilon_{s}^{0.5},\\
\label{204}&&||D\Phi_{s+1}-I||_{(r,\infty)\rightarrow(r,\infty)}\leq \epsilon_{s}^{0.5},
\end{eqnarray}
such that for
$H_{s+1}=H_{s}\circ\Phi_{s+1}=N_{s+1}+R_{s+1}$, the same assumptions as above are satisfied with `$s+1$' in place of `$s$', where $\mathcal{C}_{\eta_{s+1}}(V_{s+1})\subset\widetilde V_{s}^{-1}(\mathcal{C}_{\lambda_s}(\omega))$ and
\begin{equation}\label{206}
||\widetilde{V}_{s+1}-\widetilde{V}_{s}||_{\infty}\leq\epsilon_{s}^{0.5},
\end{equation}
\begin{equation}\label{205}
||V_{s+1}-V_{s}||_{\infty}\leq2\epsilon_{s}^{0.5}.
\end{equation}
\end{lemma}

\begin{proof}
In the step $s\rightarrow s+1$, there is saving of a factor
\begin{equation*}
e^{-\delta_{s}\left(\sum_{n}(2a_n+k_n+k'_n)|n|^{\theta}-2|n_1^*|^{\theta}\right)}\leq e^{-(2-2^\theta)\delta_{s}\sum_{i\geq3}|n_i|^{\theta}}.
\end{equation*}
Recalling after this step, we need
\begin{eqnarray*}
&&||R_{0,s+1}||_{\rho_{s+1}}^{+}\leq \epsilon_{s+1},\\
&&||R_{1,s+1}||_{\rho_{s+1}}^{+}\leq \epsilon_{s+1}^{0.6}.
\end{eqnarray*}
Consequently, in $R_{i,s}\ (i=0,1)$, it suffices to eliminate the nonresonant monomials $\mathcal{M}_{akk'}$ for which
\begin{equation*}
e^{-(2-2^\theta)\delta_{s}\sum_{i\geq3}|n_i|^{\theta}}\geq\epsilon_{s+1},
\end{equation*}
that is
\begin{equation}\label{M1}
\sum_{i\geq3}|n_i|^{\theta}\leq\frac{s^2}{(2-2^\theta)\rho}\ln\frac{1}{\epsilon_{s+1}}.
\end{equation}
On the other hand, in the small divisors analysis, we need only impose Diophantine conditions when (\ref{041803}) holds, which implies
\begin{eqnarray}\label{M2}
\nonumber\sum_{n\in\mathbb{Z}}|k_n-k_n'||n|^{\theta/2}
\nonumber&\leq& 3\cdot 6^{\theta/2}\sum_{i\geq3}|n_i|^{\theta}\ \  \mbox{(by Lemma \ref{a1})}\\
&\leq&\frac{3\cdot 6^{\theta/2}\cdot s^2}{(2-2^\theta)\rho} \ln\frac{1}{\epsilon_{s+1}}\ \  \mbox{(in view of  (\ref{M1}))}\\
&:=& B\nonumber.
\end{eqnarray}
From (\ref{M2}), we need only impose condition on $(\widetilde{V}_n)_{|n|\leq N_{*}}$, where
\begin{equation}\label{M3}
N_{*}\sim B^{2/\theta}.
\end{equation}
Correspondingly, the Diophantine condition becomes
\begin{equation}\label{M4}
\left|\left|\sum_{|n|\leq N_{*}}(k_n-k'_n)\widetilde{V}_{n,s}\right|\right|\geq \gamma\prod_{|n|\leq N_{*}}\frac{1}{1+(k_n-k'_n)^2|n|^4}.
\end{equation}

We finished the truncation step. Next we will show (\ref{M4})  preserves under small perturbation of $(\widetilde{V}_n)_{|n|\leq N_{*}}$ and this is equivalent to get lower bound on the right hand side of (\ref{M4}). More concretely, for large $N$ (will be specified latter), we have
\begin{eqnarray}
\nonumber\prod_{|n|\leq N_{*}}\frac{1}{1+(k_n-k'_n)^2|n|^4}
\nonumber&=&e^{\sum_{|n|\leq N}\ln\left(\frac{1}{1+(k_n-k'_n)^2|n|^4}\right)+\sum_{|n|> N}\ln\left(\frac{1}{1+(k_n-k'_n)^2|n|^4}\right)}\\
\nonumber&\geq& e^{-C(\theta)\sum_{|n|\leq N,k_n\neq k'_n}\ln\left(|k_n-k'_n||n|^{\frac\theta2}\right)-\sum\limits_{|n|> N,k_n\neq k'_n}4\left(|k_n-k'_n|\cdot\ln|n|\right)}\\
\nonumber&\geq& e^{-C(\theta)N\ln B-4\sum_{|n|> N,k_n\neq k'_n}\left(|k_n-k'_n||n|^{\frac{\theta}{2}}(|n|^{-\frac{\theta}{2}}\ln|n|)\right)}\\
&&\nonumber \mbox{(in view of (\ref{M2}))}\\
\nonumber&\geq& e^{-C(\theta)N\ln B-C(\theta)(N^{-\frac{\theta}{2}}\ln N)B}\ \ \mbox{(again from (\ref{M2}))}\\
\label {M6}&\geq& e^{-C(\theta)N\ln B-C(\theta)N^{-\frac{\theta}{2.5}}B}.
\end{eqnarray}
Let $$N\sim \left(\frac{B}{\ln B}\right)^{\frac{2.5}{\theta+2.5}},$$
 then by (\ref{M6}), one has
\begin{eqnarray}
\nonumber \prod_{|n|\leq N_{*}}\frac{1}{1+(k_n-k'_n)^2|n|^4}
\nonumber &\geq& e^{-C(\theta)B^{\frac{2.5}{\theta+2.5}}(\ln B)^{\frac{-\theta}{\theta+2.5}}}\\
\nonumber&\geq& e^{-C(\theta)B^{\frac{3}{\theta+3}}}\\
\nonumber&\geq& e^{-C(\theta)s^{\frac{6}{3+\theta}} (\ln{\frac{1}{\epsilon_{s+1}}})^{\frac{3}{\theta+3}}}\\
\label{M7}&>&e^{-C(\theta)(\ln{\frac{1}{\epsilon_{s+1}}})^{\frac{4}{\theta+4}}}\ \  \mbox{(for $\epsilon_0$ sufficiently small)}.
\end{eqnarray}
Assuming $V\in \mathcal{C}_{\lambda_s}(\widetilde{V}_s)$, from the lower bound (\ref{M7}), the relation (\ref{M4}) remains true if we substitute $V$ for $\widetilde{V}_s$. Moreover, there is analyticity on $\mathcal{C}_{\lambda_s}(\widetilde{V}_s)$. The transformations $\Phi_{s+1}$ is obtained as the time-1 map $X_{F_s}^{t}|_{t=1}$ of the Hamiltonian
vector field $X_{F_s}$ with $F_s=F_{0,s}+F_{1,s}$. Taking $\rho=\rho_s$, $\delta=\delta_s$ in Lemma \ref{S6}, we get
\begin{eqnarray}\label{Liu1}
||F_{i,s}||_{\rho_s+\delta_s}^{+}\leq \frac{1}{\gamma}\cdot e^{C(\theta)\delta_s^{-\frac5\theta}}||R_{i,s}||_{\rho_s}^{+},
\end{eqnarray}
where $i=0,1$. By Lemma \ref{051301}, we get
\begin{equation}\label{Liu2}
||F_{i,s}||_{\rho_s+2\delta_s}\leq\frac{C(\theta)}{\delta_s^2}||F_{i,s}||_{\rho_s+\delta_s}^{+}.
\end{equation}
Combining (\ref{200}), (\ref{201}), (\ref{Liu1}) and (\ref{Liu2}), we get
\begin{equation}\label{Liu3}
||F_{s}||_{\rho_s+2\delta_s}\leq\frac{C(\theta)}{\gamma\delta_s^2}e^{C(\theta)\delta_s^{-\frac5\theta}}(\epsilon_{s}+\epsilon_{s}^{0.6}).
\end{equation}To estimate the norm of $X_{F_s}^1=\Phi_{s+1}$, we will prove a technical lemma which establishes a connection between the norm $||\cdot||_{\rho}$ and the sup-norm $||\cdot||_{\rho,\infty}$ firstly.

\begin{lemma}\label{063004}
Given a Hamiltonian
\begin{equation}\label{050902}
H=\sum_{a,k,k'}B_{akk'}\mathcal{M}_{akk'},
\end{equation}
then for any $r>7\rho$, one has
\begin{equation}\label{050907}
\sup_{||q||_{r,\infty}<1}||X_H||_{r,\infty}\leq C(\rho,\theta)||H||_{\rho},
\end{equation}where $C(\rho,\theta)$ is a positive constant depending on $\rho$ and $\theta$ only. \color{black}\end{lemma}
\begin{proof}
In view of the definition (\ref{042501}), it suffices to estimate the upper bound for
\begin{equation*}
\left|\frac{\partial{H}}{\partial q_j}e^{r|j|^{\theta}}\right|
\end{equation*}for each $j\in\mathbb{Z}$. In view of (\ref{050902}),
one has
\begin{equation*}
\frac{\partial{H}}{\partial q_j}=\sum_{a,k,k'}B_{akk'}\left(\prod_{n\neq j}I_n(0)^{a_n}q_n^{k_n}\bar{q}_n^{k_n'}\right)\left(k_jI_j(0)^{a_j}q_j^{k_j-1}\bar{q}_j^{k_j'}\right)
\end{equation*}
and moreover
\begin{equation}\label{042503}
\left|\frac{\partial{H}}{\partial q_j}\cdot e^{r|j|^{\theta}}\right|=\left|e^{r|j|^{\theta}}\sum_{a,k,k'}B_{akk'}\left(\prod_{n\neq j}I_n(0)^{a_n}q_n^{k_n}\bar{q}_n^{k_n'}\right)\left(k_jI_j(0)^{a_j}q_j^{k_j-1}\bar{q}_j^{k_j'}\right)\right|.
\end{equation}
Based on (\ref{042602}), one has
\begin{equation}\label{050903}
|B_{akk'}|\leq ||H||_{\rho}e^{\rho(\sum_{n}(2a_n+k_n+k_n')|n|^{\theta}-2(n_1^*)^{\theta})}.
\end{equation}
Therefore, in view of $||q||_{r,\infty}<1$ and (\ref{050903}), one has
\begin{eqnarray}
|(\ref{042503})|\nonumber
&\leq&\nonumber||H||_{\rho}\left|e^{r|j|^{\theta}}\sum_{a,k,k'}k_je^{\rho(\sum_{n}(2a_n+k_n+k_n')|n|^{\theta}-2(n_1^*)^{\theta})}e^{-r(\sum_{n}(2a_n+k_n+k_n')|n|^{\theta}-|j|^{\theta})}\right|\\
&=&\nonumber||H||_{\rho}\left|\sum_{a,k,k'}k_je^{\rho(\sum_{n}(2a_n+k_n+k_n')|n|^{\theta}-2(n_1^*)^{\theta})}e^{-r(\sum_{n}(2a_n+k_n+k_n')|n|^{\theta}-2|j|^{\theta})}\right|.
\end{eqnarray}
Now we will estimate the last inequality in the following two cases:

\textbf{Case 1.} $|j|\leq n_3^*$.

Then one has
\begin{eqnarray*}
&&\left|\sum_{a,k,k'}k_je^{\rho(\sum_{n}(2a_n+k_n+k_n')|n|^{\theta}-2(n_1^*)^{\theta})}e^{-r(\sum_{n}(2a_n+k_n+k_n')|n|^{\theta}-2|j|^{\theta})}\right|\\
&\leq& \left|\sum_{a,k,k'}k_je^{\rho\sum_{i\geq 1}(n_i^*)^{\theta}}e^{-r(n_1^*)^{\theta}-r\sum_{i\geq 4}(n_i^*)^{\theta}}\right|\quad \mbox{(in view of $|j|\leq n_3^*$)}\\
&\leq&  \sum_{a,k,k'}k_je^{\frac13(-r+3\rho)\sum_{i\geq 1}(n_i^*)^{\theta}}.
\end{eqnarray*}
Note that
\begin{equation*}
\sum_{i\geq 1}(n_i^*)^{\theta}=\sum_n(2a_n+k_n+k_n')|n|^{\theta}\geq k_j.
\end{equation*}
Hence,
\begin{eqnarray}
\nonumber\sum_{a,k,k'}k_je^{\frac13(-r+3\rho)\sum_{i\geq 1}(n_i^*)^{\theta}}
\nonumber&\leq&\sum_{a,k,k'}\left(\sum_{i\geq 1}(n_i^*)^{\theta}\right)e^{\frac13(-r+3\rho)\sum_{i\geq 1}(n_i^*)^{\theta}}\\
\nonumber&\leq&\frac{12}{r-3\rho}\sum_{a,k,k'}e^{\frac14(-r+3\rho)\sum_{i\geq 1}(n_i^*)^{\theta}}\quad\mbox{(in view of (\ref{042608}))}\\
\nonumber&=&\frac{12}{r-3\rho}\sum_{a,k,k'}
e^{\frac14(-r+3\rho)\sum_{n}(2a_n+k_n+k'_n)|n|^{\theta}}\\
\nonumber&\leq&\frac{12}{r-3\rho}\left(\sum_{a}
e^{\frac14(-r+3\rho)\sum_{n}2a_n|n|^{\theta}}\right)
\left(\sum_{k}e^{\frac14(-r+3\rho)\sum_{n}k_n|n|^{\theta}}\right)^2
\\
\nonumber&\leq&\frac{12}{r-3\rho}\prod_{n\in\mathbb{Z}}\left({1-e^{\frac12(-r+3\rho)|n|^{\theta}}}\right)^{-1}\prod_{n\in\mathbb{Z}}
\left({1-e^{\frac14(-r+3\rho)|n|^{\theta}}}\right)^{-2}\\
\nonumber&&\mbox{(in view of (\ref{041809}))}\\
\nonumber&\leq&\frac{12}{\rho}\prod_{n\in\mathbb{Z}}\left({1-e^{-\rho|n|^{\theta}}}\right)^{-1}\prod_{n\in\mathbb{Z}}
\left({1-e^{-\rho|n|^{\theta}}}\right)^{-2}\\
\nonumber&&  \ {\mbox{(in view of $r>7\rho$)}}    \\
&\leq&\label{050904}\frac{12}{\rho}\left(\frac{1}{\rho}\right)^{C_1{(\theta)}{\rho^{-\frac{1}{\theta}}}},
\end{eqnarray}
where the last inequality is based on (\ref{0418010}) and $C_1(\theta)$ is a positive constant depending on $\theta$ only.

\textbf{Case 2.} $|j|> n_3^*$, which implies $k_j\leq 2$.

Then one has
\begin{eqnarray}
\nonumber&&\left|\sum_{a,k,k'}k_je^{\rho(\sum_{n}(2a_n+k_n+k_n')|n|^{\theta}-2(n_1^*)^{\theta})}e^{-r(\sum_{n}(2a_n+k_n+k_n')|n|^{\theta}-2|j|^{\theta})}\right|\\
\nonumber&\leq&2 \left|\sum_{a,k,k'}e^{\rho\sum_{i\geq 3}(n_i^*)^{\theta}}e^{-r\sum_{i\geq 3}(n_i^*)^{\theta}}\right|\quad{(\mbox{in view of (\ref{001}) and $k_j\leq2$})}\\
\label{201606032}&=& 2\left|\sum_{a,k,k'}e^{(-r+\rho)\sum_{i\geq 3}(n_i^*)^{\theta}}\right|.
\end{eqnarray}
If $\{n_i\}_{i\geq 1}$ is given, then $\{2a_n+k_n+k'_n\}_{n\in\mathbb{Z}}$ is specified, and hence $(a,k,k')$ is specified up to a factor of
$$\prod_{n}(1+l_n^2),$$
where
$$l_n=\#\{j:n_j=n\}.$$
Since $|j|>n_3^*$, then $j\in\{n_1,n_2\}$. Hence, if $(n_i)_{i\geq 3}$ and $j$ are given, then $n_1$ and $n_2$ are uniquely determined. Then, one has
\begin{eqnarray}\label{050905}
\nonumber(\ref{201606032})&\leq&\nonumber 2\left|\sum_{(n_i)_{i\geq3}}\prod_{|n|\leq n_1^*}(1+l_n^2)e^{(-r+\rho)\sum_{i\geq 3}(n_i^*)^{\theta}}\right| \\
\nonumber&\leq&10\left|\left(\sum_{(n_i)_{i\geq3}}e^{-5\rho\sum_{i\geq 3}(n_i^*)^{\theta}}\right)\sup_{(n_i)_{i\geq3}}\left(\prod_{|n|\leq n_1^*}(1+l_n^2)e^{-\rho\sum_{i\geq 3}(n_i^*)^{\theta}}\right)\right|\\
\nonumber&\leq&10\left(\frac{1}{\rho}\right)^{C_2{(\theta)}{\rho^{-\frac{1}{\theta}}}}\sum_{(n_i)_{i\geq3}}e^{-5\rho\sum_{i\geq 3}(n_i^*)^{\theta}}\ \ \mbox{(in view of (\ref{201606031}))}\\
\nonumber&=&10\left(\frac{1}{\rho}\right)^{C_2{(\theta)}{\rho^{-\frac{1}{\theta}}}}\sum_{(l_n)_{|n|\leq n_3^*}}e^{-5\rho\sum_{|n|\leq n_3^*}l_n|n|^{\theta}}\\
\nonumber&\leq&10\left(\frac{1}{\rho}\right)^{C_2{(\theta)}{\rho^{-\frac{1}{\theta}}}}\prod_{n\in\mathbb{Z}}\left(1-e^{-5\rho|n|^{\theta}}\right)^{-1}\ \ \mbox{(in view of (\ref{041809}))}\\
\nonumber&\leq&\left(\frac{1}{\rho}\right)^{C_3{(\theta)}{\rho^{-\frac{1}{\theta}}}}\ \ \qquad\qquad\qquad\qquad\qquad\mbox{(in view of (\ref{0418010}))},
\end{eqnarray}
where $C_2(\theta),C_{3}(\theta)$ are positive constants depending on $\theta$ only.

 Hence, we finished the proof of (\ref{050907}).
\end{proof}
By Lemma \ref{063004}, we get
\begin{eqnarray}\label{Liu4}
\sup_{||q||_{r,\infty}<1}||X_{F_s}||_{r,\infty}\nonumber
&\leq&C(\rho,\theta)||F_{s}||_{\rho_s+2\delta_s}\nonumber\\
&\leq&\frac{C(\rho,\theta)}{\gamma\delta_s^2}e^{C(\theta)\delta_s^{-\frac5\theta}}(\epsilon_{s}+\epsilon_{s}^{0.6})\nonumber\\
&\leq&\epsilon_{s}^{0.55},
\end{eqnarray}
where noting that $0<\epsilon_0\ll1$ small enough and depending on $\rho,\theta$ only.

Since $\epsilon_{s}^{0.55}\ll\frac{1}{\pi^2(s+1)^2}=d_{s+1}-d_s$, we have $\Phi_{s+1}:D_{s+1}\rightarrow D_{s}$ with
\begin{equation}\label{Liu5}
\|\Phi_{s+1}-id\|_{r,\infty}\leq\sup_{q\in D_s}\|X_{F_s}\|_{r,\infty}\leq\epsilon_{s}^{0.55}<\epsilon_{s}^{0.5},
\end{equation}
which is the estimate (\ref{203}). Moreover, from (\ref{Liu4}) we get
\begin{equation}\label{Liu6}
\sup_{q\in D_s}||DX_{F_s}-I||_{r,\infty}\leq\frac{1}{d_s}\epsilon_{s}^{0.55}\ll\epsilon_{s}^{0.5},
\end{equation}
and thus the estimate (\ref{204}) follows.

Moreover, under the assumptions (\ref{200})-(\ref{202}) at stage $s$, we get from (\ref{N20}), (\ref{N21}) and (\ref{N22}) that
\begin{eqnarray*}
||R_{0,s+1}||_{\rho_{s+1}}^{+}
&\leq& e^{\frac{3s^{\frac{4}{\theta}+6}}{\tau^{\frac{2}{\theta}+3}}}
\left(\epsilon_{0}^{(\frac{3}{2})^s}+\epsilon_{0}^{0.9(\frac{3}{2})^{s-1}}\right)\left(\epsilon_{0}^{(\frac{3}{2})^s}+\epsilon_{0}^{1.8(\frac{3}{2})^{s-1}}\right)\\
&=&e^{\frac{3s^{\frac{4}{\theta}+6}}{\tau^{\frac{2}{\theta}+3}}}\left(\epsilon_{0}^{2.2(\frac{3}{2})^s}+\epsilon_{0}^{1.8(\frac{3}{2})^s}
+\epsilon_{0}^{1.6(\frac{3}{2})^s}+\epsilon_{0}^{2(\frac{3}{2})^s}\right)\\
&\leq& 4e^{\frac{3s^{\frac{4}{\theta}+6}}{\tau^{\frac{2}{\theta}+3}}}\epsilon_{0}^{1.6(\frac{3}{2})^s}\\
&<&\epsilon_{0}^{1.5(\frac{3}{2})^s}\ \mbox{for $0<\epsilon_0\ll1$ (depending on $\rho,\theta$ only)}\\
&=&\epsilon_{s+1},\\
||R_{1,s+1}||_{\rho_{s+1}}^{+}
&\leq& e^{\frac{3s^{\frac{4}{\theta}+6}}{\tau^{\frac{2}{\theta}+3}}}\left( \epsilon_{0}^{(\frac{3}{2})^s}+\epsilon_{0}^{1.8(\frac{3}{2})^{s-1}}\right)\\
&=&e^{\frac{3s^{\frac{4}{\theta}+6}}{\tau^{\frac{2}{\theta}+3}}}\left( \epsilon_{0}^{(\frac{3}{2})^s}+\epsilon_{0}^{1.2(\frac{3}{2})^{s}}\right)\\
&\leq&2 e^{\frac{3s^{\frac{4}{\theta}+6}}{\tau^{\frac{2}{\theta}+3}}}\epsilon_{0}^{(\frac{3}{2})^s}\\
&<&\epsilon_{s+1}^{0.6}\ \ \mbox{for $0<\epsilon_0\ll1$ (depending on $\rho,\theta$ only)},\\
\end{eqnarray*}
and
\begin{eqnarray*}
||R_{2,s+1}||_{\rho_{s+1}}^{+}
&\leq& ||R_{2,s}||_{\rho_{s}}^{+}+e^{\frac{3s^{\frac{4}{\theta}+6}}{\tau^{\frac{2}{\theta}+3}}}
\left(\epsilon_{0}^{(\frac{3}{2})^s}+\epsilon_{0}^{0.6(\frac{3}{2})^{s}}\right)\\
&\leq&(1+d_s)\epsilon_0+2e^{\frac{3s^{\frac{4}{\theta}+6}}{\tau^{\frac{2}{\theta}+3}}}
\epsilon_{0}^{0.6(\frac{3}{2})^s}\\
&\leq&(1+d_{s+1})\epsilon_0\ \ \mbox{for $0<\epsilon_0\ll1$ (depending on $\rho,\theta$ only)},
\end{eqnarray*}
which are just the assumptions (\ref{200})-(\ref{202}) at stage $s+1$.

Next, let $S=\mathcal{C}_{\frac{1}{10}\lambda_s\eta_s}(V_s)$  and if $V\in \mathcal{C}_{\frac{\eta_s}{2}}(V_s)$, by using Cauchy's estimate implies
\begin{eqnarray}
\nonumber\sum_{n\in\mathbb{Z}}\left|\frac{\partial \widetilde{V}_{m,s}}{\partial V_n}(V)\right|
\nonumber&\leq& \frac{2}{\eta_s}||\widetilde{V}_s||_\infty\\
\label{M11}&<&10 \eta_s^{-1}\ \ \mbox{(since $||\widetilde{V}_s||_\infty\leq 1 $)} \ \mbox{for all $m$},
\end{eqnarray}
and let $X\in \mathcal{C}_{\frac{1}{10}\lambda_s\eta_s}(V_s)$, then
\begin{eqnarray*}
||\widetilde{V}_s(X)-\omega||_{\infty}
&=&||\widetilde{V}_s(X)-\widetilde{V}_s(V_s)||_{\infty}\\
& \leq&\sup_{\mathcal{C}_{\frac{1}{10}\lambda_s\eta_s}(V_s)}\left|\left|\frac{\partial \widetilde{V}_s}{\partial V}\right|\right|_{l^{\infty}\rightarrow l^{\infty}}||X-V_s||_{\infty}\\
&<&10 \eta_s^{-1}\cdot\frac{1}{10}\lambda_s\eta_s\ \ \mbox{(in view of (\ref{M11}))}\\
&=&\lambda_s,
\end{eqnarray*}
that is
\begin{equation*}
\widetilde{V}_s(\mathcal{C}_{\frac{1}{10}\lambda_s\eta_s}(V_s))\subseteq \mathcal{C}_{\lambda_s}(\omega).
\end{equation*}
Recalling the estimates in section 4, we have
\begin{eqnarray}
\nonumber|B^{(m)}_{a00}|
\nonumber&\leq& ||R_{1,s+1}||_{\rho_{s+1}}^+e^{2\rho_{s+1}(\sum_{n}a_n|n|^{\theta}+|m|^{\theta}-(n_1^{*})^{\theta})}\\
\label{M12}&<&\epsilon_{0}^{0.6(\frac{3}{2})^{s}}e^{2\rho_{s+1}(\sum_{n}a_n|n|^{\theta}+|m|^{\theta}-(n_1^{*})^{\theta})}.
\end{eqnarray}
Assuming further
\begin{equation}\label{M13}
I_{n}(0)\leq e^{-2r|n|^{\theta}}
\end{equation}
and
\begin{equation}\label{M14}
\rho_s<\frac{1}{2}r,\forall s,
\end{equation}
we obtain
\begin{eqnarray}
\nonumber|\sum_{a}B^{(m)}_{a00}\mathcal{M}_{a00}|
\nonumber&\leq & \epsilon_{0}^{0.6(\frac{3}{2})^{s}}\sum_{a}e^{2\rho_{s+1}(\sum_{n}a_n|n|^{\theta}+|m|^{\theta}-(n_1^{*})^{\theta})}\prod_{n}I_{n}(0)^{a_n}\\
\nonumber&\leq& \epsilon_{0}^{0.6(\frac{3}{2})^{s}}\sum_{a}e^{2\rho_{s+1}(\sum_{n}a_n|n|^{\theta})}\prod_{n}I_{n}(0)^{a_n}\\
\nonumber&\leq& \epsilon_{0}^{0.6(\frac{3}{2})^{s}}\sum_{a}e^{\sum_{n}2\rho_{s+1}a_n|n|^{\theta}-\sum_{n}2r a_n|n|^{\theta}}\ \mbox{(in view of (\ref{M13}))}\\
\nonumber&\leq& \epsilon_{0}^{0.6(\frac{3}{2})^{s}}\sum_{a}e^{-r(\sum_{n}a_n|n|^{\theta})}\ \mbox{(in view of (\ref{M14}))}\\
\nonumber&\leq& \epsilon_{0}^{0.6(\frac{3}{2})^{s}}\prod_{n}(1-e^{-r n^{\theta}})^{-1} \ \mbox{(by Lemma \ref{a3})}\\
\label{M15}&\leq&\left(\frac{1}{r}\right)^{C(\theta){r^{-\frac{1}{\theta}}}}\epsilon_{0}^{0.6(\frac{3}{2})^{s}}\ \ \mbox{(by Lemma \ref{a5})}.
\end{eqnarray}
By (\ref{M15}), we have
\begin{eqnarray}
\nonumber|\widetilde{V}_{m,s+1}-\widetilde{V}_{m,s}|
\nonumber&<&\left(\frac{1}{r}\right)^{C(\theta){r^{-\frac{1}{\theta}}}}\epsilon_{0}^{0.6(\frac{3}{2})^{s}}\\
\label{M16}&<&\epsilon_{s}^{0.5}\ \ \mbox{(for $\epsilon_0$ small enough)},
\end{eqnarray}
which verifies (\ref{206}). Further applying Cauchy's estimate on $\mathcal{C}_{\lambda_s\eta_s}(V_s)$, one gets
\begin{eqnarray}\label{633}
\nonumber\sum_{n\in\mathbb{Z}}\left|\frac{\partial \widetilde{V}_{m,s+1}}{\partial V_n}-\frac{\partial \widetilde{V}_{m,s}}{\partial V_n}\right|
\nonumber&\leq& C(\theta)\frac{||\widetilde{V}_{s+1}-\widetilde{V}_{s}||_\infty}{\lambda_s\eta_s}\\
\nonumber&\leq& C(\theta)\frac{\epsilon_{s}^{0.5}}{\lambda_s\eta_s}\\
\nonumber&\leq& e^{C(\theta)(\ln\frac{1}{\epsilon_{s+1}})^{\frac{4}{4+\theta}}-\frac13\ln\frac{1}{\epsilon_{s+1}}}\left(\frac{1}{\eta_s}\right)\\
\nonumber&\leq& e^{-\frac14\ln\frac{1}{\epsilon_{s+1}}}\left(\frac{1}{\eta_s}\right)\ \ \mbox{(for $\epsilon_0$ small enough)}\\
\label{M17}&=& \frac{1}{\eta_s}\epsilon_{0}^{\frac{1}{4}(\frac{3}{2})^{s+1}}.
\end{eqnarray}
Since
\begin{equation*}
\eta_{s+1}=\frac{1}{20}\lambda_s\eta_s,
\end{equation*}
it follows that
\begin{eqnarray}
\nonumber\eta_{s+1}&\geq& \eta_s e^{-C(\theta)(\ln\frac{1}{\epsilon_0})^{\frac{4}{4+\theta}}(\frac32)^{\frac{4}{4+\theta}(s+1)}}\\
\nonumber&\geq& \eta_se^{-C(\theta)\ln\frac{1}{\epsilon_0}\cdot(\frac32)^{\frac{5}{5+\theta}s}}\ \ \ \mbox{(for $\epsilon_0$ small enough)}\\
\label{M18}&=&\eta_s\epsilon_0^{C(\theta)(\frac32)^{\frac{5s}{5+\theta}}},
\end{eqnarray}
and hence by iterating (\ref{M18}) implies
\begin{eqnarray}
\nonumber\eta_{s}&\geq&\eta_0\epsilon_{0}^{C(\theta)\sum_{i=0}^{s-1}(\frac{3}{2})^{\frac{5i}{\theta+5}}}\\
\nonumber&=&\eta_0\epsilon_{0}^{C(\theta)\frac{(\frac32)^{\frac{5s}{\theta+5}}-1}{(\frac32)^{\frac{5}{\theta+5}}-1}}\\
\nonumber&>&\epsilon_{0}^{C(\theta)(\frac{3}{2})^{\frac{5s}{\theta+5}}}\\
\label{M19}&\geq&\epsilon_{0}^{\frac{1}{100}(\frac{3}{2})^{s}}\ \ \ \mbox{(for $\epsilon_0$ small enough)}.
\end{eqnarray}
On $ \mathcal{C}_{\frac{1}{10}\lambda_s\eta_s}(V_s)$, for any $m$, we deduce from (\ref{M17}), (\ref{M19}) and the assumption (\ref{199}) that
\begin{eqnarray*}
\sum_{n\in\mathbb{Z}}\left|\frac{\partial \widetilde{V}_{m,s+1}}{\partial V_n}-\delta_{mn}\right|
&\leq&\sum_{n\in\mathbb{Z}}\left|\frac{\partial \widetilde{V}_{m,s+1}}{\partial V_n}-\frac{\partial \widetilde{V}_{m,s}}{\partial V_n}\right|+\sum_{n\in\mathbb{Z}}\left|\frac{\partial \widetilde{V}_{m,s}}{\partial V_n}-\delta_{mn}\right|\\
&\leq&\epsilon_{0}^{(\frac{3}{8}-\frac{1}{100})(\frac{3}{2})^{s}}+d_s\epsilon_{0}^{\frac{1}{10}}\\
&<&d_{s+1}\epsilon_{0}^{\frac{1}{10}},
\end{eqnarray*}
and consequently
\begin{equation}\label{M20}
\left|\left|\frac{\partial \widetilde{V}_{s+1}}{{\partial V}}-I\right|\right|_{l^{\infty}\rightarrow l^{\infty}}<d_{s+1}\epsilon_{0}^{\frac{1}{10}},
\end{equation}
which verifies (\ref{199}) for $s+1$.

Finally, we will freeze $\omega$ by invoking an inverse function theorem. Consider the following functional equation
\begin{equation}\label{M21}
\widetilde{V}_{s+1}(V_{s+1})=\omega,  V_{s+1}\in \mathcal{C}_{\frac{1}{10}\lambda_s\eta_s}(V_s),
\end{equation}
from (\ref{M20}) and the standard inverse function theorem implies (\ref{M21}) having a solution $V_{s+1}$, which verifies (\ref{198}) for $s+1$. Rewriting (\ref{M21}) as
\begin{equation}\label{M22}
V_{s+1}-V_s=(I-\widetilde{V}_{s+1})(V_{s+1})-(I-\widetilde{V}_{s+1})({V_s})+(\widetilde{V}_s-\widetilde{V}_{s+1})(V_s),
\end{equation}
and by using (\ref{M16}) (\ref{M20}) implies
\begin{equation}\label{M23}
||V_{s+1}-V_s||_{\infty}\leq (1+d_{s+1})\epsilon_{0}^{\frac{1}{10}}||V_{s+1}-V_s||_{\infty}+\epsilon_s^{0.5}<2\epsilon_s^{0.5}\ll \lambda_s\eta_s,
\end{equation}
which verifies (\ref{205}) and completes the proof of the iterative lemma.
\end{proof}

We are now in a position to prove the convergence. To apply iterative lemma with $s=0$, set
\begin{equation*}
V_0=\omega,\hspace{12pt}\widetilde{V}_0=id,\hspace{12pt}\eta_0=1-\sup_{n\in\mathbb{Z}}|\omega_n|,\hspace{12pt}\rho_0=\frac{r}{20},\hspace{12pt}\epsilon_0=C\epsilon,
\end{equation*}
and consequently (\ref{198})-(\ref{202}) with $s=0$ are satisfied. Hence, the iterative lemma applies, and we obtain a decreasing
sequence of domains $D_{s}\times\mathcal{C}_{\eta_{s}}(V_{s})$ and a sequence of
transformations
\begin{equation*}
\Phi^s=\Phi_1\circ\cdots\circ\Phi_s:\hspace{6pt}D_{s}\times\mathcal{C}_{\eta_{s}}(V_{s})\rightarrow D_{0}\times\mathcal{C}_{\eta_{0}}(V_{0}),
\end{equation*}
such that $H\circ\Phi^s=N_s+P_s$ for $s\geq1$. Moreover, the
estimates (\ref{203})-(\ref{205}) hold. Thus we can show $V_s$ converge to a limit $V_*$ with the estimate
\begin{equation*}
||V_*-\omega||_{\infty}\leq\sum_{s=0}^{\infty}2\epsilon_{s}^{0.5}<\epsilon_{0}^{0.4},
\end{equation*}
and $\Phi^s$ converge uniformly on $D_*\times\{V_*\}$, where $D_*=\{(q_n)_{n\in\mathbb{Z}}:\frac{2}{3}\leq|q_n|e^{r|n|^{\theta}}\leq\frac{5}{6}\}$, to $\Phi:D_*\times\{V_*\}\rightarrow D_0$ with the estimates
\begin{eqnarray}
\nonumber&&||\Phi-id||_{r,\infty}\leq \epsilon_{s}^{0.4},\\
\nonumber&&||D\Phi-I||_{(r,\infty)\rightarrow(r,\infty)}\leq \epsilon_{s}^{0.4}.
\end{eqnarray}
Hence
\begin{equation}\label{060101}
H_*=H\circ\Phi=N_*+R_{2,*},
\end{equation}
where
\begin{equation}
N_*=\sum_{n\in\mathbb{Z}}(n^2+\omega_n)|q_n|^2
\end{equation}
and
\begin{equation}\label{062811}
||R_{2,*}||_{\frac{r}{10}}^{+}\leq\frac{7}{6}\epsilon_0.
\end{equation}
By (\ref{050907}), the Hamiltonian vector field $X_{R_{2,*}}$ is a bounded map from $\mathfrak{H}_{r,\infty}$ into $\mathfrak{H}_{r,\infty}$. Taking \begin{equation}\label{072701}
I_n(0)=\frac{3}{4}e^{-2r|n|^{\theta}},
 \end{equation}we get an invariant torus $\mathcal{T}$ with frequency $(n^2+\omega_n)_{n\in\mathbb{Z}}$ for ${X}_{H_*}$. Finally, by $X_H\circ\Phi=D\Phi\cdot{X}_{H_*}$, $\Phi(\mathcal{T})$ is the desired invariant torus for the NLS (\ref{L1}). Moreover, we deduce the torus $\Phi(\mathcal{T})$ is linearly stable from the fact that (\ref{060101}) is a normal form of order 2 around the invariant torus.

\section{Long time stability of full dimensional KAM tori}
In this section, we would like to study the long time stability of the invariant torus $\Phi(\mathcal{T})$. To this end, we will construct a normal form of order $M$ in the neighborhood of $\mathcal{T}=(I_n(0)=\frac{3}{4}e^{-2r|n|^{\theta}})_{n\in\mathbb{Z}}$ firstly (see (\ref{072701})). Precisely,  given a $\tau\ll\epsilon$, define the neighborhood of the torus $\mathcal{T}$ by
\begin{equation}
\widetilde D{(\tau)}=\{q:||J||_{r,\infty}<\tau, \mbox{where}\ J=I-I(0)\}
\end{equation}
and one has the following theorem:

\begin{theorem}\label{060102}
Consider the normal form of order $2$ (see (\ref{060101}))
\begin{equation*}
H_*=H\circ\Phi=N_*+R_{2,*}.
\end{equation*}
If $\omega=(\omega_n)_{n\in\mathbb{Z}}$ is Diophantine, then there is
a symplectic map $$\Psi:\widetilde D(\tau)\rightarrow \widetilde D(2\tau)$$ such that
\begin{equation}
H_*\circ\Psi=N_*+Z+Q,
\end{equation}
where $Z$ is the integrable term depending on the variables $I$ and the remainder term $Q$ satisfies
\begin{equation}\label{083104}
\sup_{\widetilde D(\tau)}||X_{Q_{(M+1)}}||_{r,\infty}\leq \tau^{\frac14|\ln \tau|^{\frac{\theta}{10}}}.
\end{equation}
\end{theorem}
Firstly, we will construct a normal form of order $3$ around the tori $\mathcal{T}$ based on a standard normal form procedure where noting $H_*$ is already a normal form of order $2$. Firstly, rewrite $R_{2,*}$ as
\begin{equation*}
R_{2,*}=\sum_{j\geq 2}R_j
\end{equation*}
with
\begin{equation*}
R_j=\sum_{|l|=j}J^{l}\sum_{a,k,k'}R_{2,*;akk'}^{(l)}\mathcal{M}_{akk'}.
\end{equation*}
Then we have the following lemma
\begin{lemma}\label{080301}
Consider the normal form of order $2$ (see (\ref{060101}))
\begin{equation*}
H_*=H\circ\Phi=N_*+R_{2,*}.
\end{equation*}
If $\omega=(\omega_n)_{n\in\mathbb{Z}}$ is Diophantine,
then there exists a symplectic map $\Psi_2=X_{F_2}^t|_{t=1}$ such that
\begin{eqnarray}
H_{3}:=H_*\circ \Psi_2=N_{*}+Z_3+Q_{3},
\end{eqnarray}where
\begin{equation}\label{080304}
Z_3=\sum_{|l|=2}J^{l}\sum_{a,k,k'\atop k=k'}R_{2,*;akk'}^{(l)}\mathcal{M}_{akk'},
\end{equation}
\begin{equation}
Q_{3}=\sum_{j\geq 3}Q_{3j},
\end{equation}
with
\begin{equation}
Q_{3j}=\sum_{|l|=j}J^{l}\sum_{a,k,k'}Q_{3;akk'}^{(l)}\mathcal{M}_{akk'}.
\end{equation}
Moreover, the following estimates hold:
\begin{equation}\label{080302}
||F_{2}||_{\frac{r}{10}+\frac{r}{120}}\leq C_1(r,\theta,\gamma)\cdot||R_{2,*}||_{\frac r{10}+\frac{r}{240}},
\end{equation}
and
\begin{equation}\label{080303}
||Q_{3j}||_{\frac{r}{10}+\frac{r}{80}}\leq ||R_{2,*}||_{\frac{r}{10}+\frac{r}{240}}\left( C_2(r,\theta,\gamma) ||R_{2,*}||_{\frac{r}{10}+\frac{r}{240}}\right)^{j-2} ,
\end{equation}
where $C_1(r,\theta,\gamma)$ and $C_2(r,\theta,\gamma)$ are positive constants depending on $r,\theta, \gamma$ only.
\end{lemma}
\begin{proof}
\textbf{Step 1. The derivative of the homological equation}

Let
\begin{equation*}
F_{2}=\sum_{|l|=2}J^{l}\sum_{a,k,k'}F_{2;akk'}^{(l)}\mathcal{M}_{akk'},
\end{equation*}
and let $\Psi_{2}=X_{F_{2}}^t|_{t=1}$ be the time-1 map of the Hamiltonian vector field $X_{F_{2}}$.

Using Taylor's formula,
\begin{eqnarray*}
H_{3}&:=&H_{*}\circ X_{F_{2}}^t|_{t=1}\\
&=&(N_*+R_{2,*})\circ X_{F_{2}}^t|_{t=1}\\
&=&N_*+\{N_*,F_{2}\}+\sum_{n\geq2}\frac{1}{n!}\underbrace{\{\cdots\{N_*,{F_{2}}\},{F_{2}},\cdots,{F_{2}}\}}_{n-\mbox{fold}}\\
&&+R_{2}+\sum_{n\geq1}\frac{1}{n!}\underbrace{\{\cdots\{R_{2},{F_{2}}\},{F_{2}},\cdots,{F_{2}}\}}_{n-\mbox{fold}}\\
&&+\sum_{n\geq0}\frac{1}{n!}\underbrace{\left\{\cdots\left\{\sum_{j\geq 3}R_{j},{F_{2}}\right\},{F_{2}},\cdots,{F_{2}}\right\}}_{n-\mbox{fold}}.
\end{eqnarray*}
Then we obtain the homological equation
\begin{equation}\label{080305}
\{N_*,F_{2}\}+R_{2}={Z_{3}},
\end{equation}
where $Z_3$ is given by (\ref{080304}).

If the homological equation (\ref{080305}) is solvable, then we define

\begin{eqnarray}
\nonumber Q_{3}
&=&\label{080306}\sum_{n\geq2}\frac{1}{n!}\underbrace{\{\cdots\{N_*,{F_{2}}\},{F_{2}},\cdots,{F_{2}}\}}_{n-\mbox{fold}}\\
&&\label{080307}+\sum_{n\geq1}\frac{1}{n!}\underbrace{\{\cdots\{R_{2},{F_{2}}\},{F_{2}},\cdots,{F_{2}}\}}_{n-\mbox{fold}}\\
&&\label{080308}+\sum_{n\geq0}\frac{1}{n!}\underbrace{\left\{\cdots\left\{\sum_{j\geq 3}R_j,{F_{2}}\right\},{F_{2}},\cdots,{F_{2}}\right\}}_{n-\mbox{fold}}
\end{eqnarray}
and one has
\begin{equation*}
H_{3}=N_*+Z_{3}+Q_{3}.
\end{equation*}

\textbf{Step 2. The solution of the homological equation (\ref{080305}).}
It is easy to show that the solution of the homological equation is given by
\begin{equation}
F_{2;akk'}^{(l)}=\frac{R_{2,*;akk'}^{(l)}}{\sum_n(k_n-k'_n)(n^2+\omega_n)}.
\end{equation}
Moreover, the inequality (\ref{080302}) holds in view of the fact that $\omega$ is Diophantine and following the proof of Lemma \ref{S6}.

\textbf{Step 3. Estimate the remainder terms $Q_{3}$.}

To this end, we will estimate the norm of (\ref{080306})-(\ref{080308}) respectively.
Without loss of generality, we only consider the following term
\begin{equation}\label{080309}
\frac{1}{n!}\underbrace{\{\cdots\{R_2,{F_{2}}\},{F_{2}},\cdots,{F_{2}}\}}_{n-\mbox{fold}},
\end{equation}
which is in (\ref{080307}). By a direct calculation, one has
\begin{eqnarray}
&&\nonumber\left|\left|\frac{1}{n!}\underbrace{\{\cdots\{R_2,{F_{2}}\},{F_{2}},\cdots,{F_{2}}\}}_{n-\mbox{fold}}\right|\right|_{\frac{r}{10}+\frac{r}{80}}\\
&\leq&\nonumber\frac{1}{n!}\left(C(r,\theta)
||F_{2}||_{\frac{r}{10}+\frac{r}{120}}\right)^n\left(\frac{n}{r}\right)^n||R_2||_{\frac{r}{10}+\frac{r}{240}}\\
&&\nonumber(\mbox{following the proof of (\ref{3.3}) and $C(r,\theta)$ is a positive constant depending on $r$ and $\theta$ only})\\
&\leq&\nonumber\frac{1}{n!}\left(C(r,\theta)\cdot C_1(r,\theta,\gamma)\right)^n\left(\frac{n}{r}\right)^n||R_2||_{\frac{r}{10}+\frac{r}{240}}^{n+1}\\
&&(\mbox{in view of the first inequality in (\ref{080302})})\nonumber\\
&\leq&\nonumber\left(\frac{e}{r}\cdot C(r,\theta)\cdot C_1(r,\theta,\gamma)\right)^n||R_2||_{\frac{r}{10}+\frac{r}{240}}^{n+1}\qquad (\mbox{in view of $n^n/n!\leq e^n$})\\
&\leq&\nonumber\left( C_2(r,\theta,\gamma)\right)^n ||R_{2,*}||_{\frac{r}{10}+\frac{r}{240}}^{n+1},
\end{eqnarray}
where
\begin{equation*}
C_2(r,\theta,\gamma)=\frac{e}{r}\cdot C(r,\theta)\cdot C_1(r,\theta,\gamma)
\end{equation*}depends on $r,\theta,\gamma$ and noting
 \begin{equation*}
 ||R_2||_{\frac{r}{10}+\frac{r}{240}}\leq||R_{2,*}||_{\frac{r}{10}+\frac{r}{240}}
 \end{equation*}
 Finally, note that the term (\ref{080309}) contains at least $(2+n)$ $J's$ and we finish the proof of (\ref{080303}).
\end{proof}
\begin{remark}

In view of  (\ref{N7}), (\ref{062811}) and (\ref{080302}), one has
\begin{equation*}
||F_{2}||_{\frac{r}{10}+\frac{r}{120}}\leq \widetilde C(r,\theta,\gamma)\epsilon_0,
\end{equation*}
where $\widetilde C(r,\theta,\gamma)$ is a positive constant depending on $r,\theta,\gamma$ only. For convenience, we choose $\epsilon_0$ small enough at the beginning such that
\begin{equation*}
\widetilde C(r,\theta,\gamma)\epsilon_0\leq1,
\end{equation*}
and one has
\begin{equation}\label{081001}
2C(\rho,\theta)_{\rho=\frac{r}{10}+\frac{r}{120}}||F_{2}||_{\frac{r}{10}+\frac{r}{120}}\leq 1,
\end{equation}
where the positive constant $C(\rho,\theta)$ is given in (\ref{050907}).
Similarly, one has
\begin{equation*}
||Q_{3j}||_{\frac{r}{10}+\frac{r}{80}}\leq 1,
\end{equation*}
where by choosing $\epsilon_0$ smaller depending on $r,\theta,\gamma$ only.
Furthermore, one has
\begin{equation*}
||Q_3||_{\frac{r}{10}+\frac{r}{80}}\leq \sup_{j}||Q_{3j}||_{\frac{r}{10}+\frac{r}{80}}\leq 1.
\end{equation*}
\end{remark}

Given a large $M>0$, now we will construct a normal form of order $M$ around the torus $\mathcal{T}$. To this end, we give an iterative lemma first.
Take
\begin{equation}\label{062809}
\delta=\frac{r}{80M}.
\end{equation}
For $s\geq 3$, denote
\begin{equation}\label{062810}
\rho_{s}=\frac{r}{10}+\frac{r}{80}+2(s-3)\delta.
\end{equation}
In view of the following two constants
\begin{equation}\label{081002}
C_1(\delta,\theta,\gamma)=\frac{e^3}{\gamma}\cdot e^{C(\theta)\delta^{-\frac5\theta}}
\end{equation}and
\begin{equation}\label{081003}
C_2(\delta,\theta)=\frac{1}{\delta}\left(\frac{1}{\delta}\right)^{C({\theta}){\delta^{-\frac{1}{\theta}}}},
\end{equation}
which are given in Lemma \ref{S6} (see (\ref{S7})) and Lemma \ref{010} (see (\ref{042704})) respectively, define
\begin{equation}\label{062806}
C(\delta,\theta,\gamma)=\frac{e}{\delta}\cdot C_{1}(\delta,\theta,\gamma)\cdot C_2(\delta,\theta).
\end{equation}
\begin{lemma}\label{062808}
Consider the normal form of order $s$ ($s\geq 3$)
\begin{equation}
H_{s}=N_*+Z_{s}+Q_{s},
\end{equation}
where
\begin{equation}\label{060301}
Z_{s}=\sum_{2\leq j\leq s-1}Z_{sj},
\end{equation}
\begin{equation}
Q_{s}=\sum_{j\geq s}Q_{sj},
\end{equation}
with
\begin{equation}
Z_{sj}=\sum_{|l|=j}J^{l}\sum_{a,k,k'\atop k=k'}Z_{s;akk'}^{(l)}\mathcal{M}_{akk'}
\end{equation}
and
\begin{equation}
Q_{sj}=\sum_{|l|=j}J^{l}\sum_{a,k,k'}Q_{s;akk'}^{(l)}\mathcal{M}_{akk'}.
\end{equation}
When $s\geq 3$, suppose $Z_{sj}$ and $Q_{sj}$ satisfy the following estimates
\begin{equation}\label{063001}
||Z_{sj}||_{\rho_{s}}\leq (C(\delta,\theta,\gamma))^{(s-3)j}
\end{equation}
and
\begin{equation}\label{062801}
||Q_{sj}||_{\rho_{s}}\leq (C(\delta,\theta,\gamma))^{(s-3)j}.
\end{equation}

Then there exists a symplectic map $\Psi_s=X_{F_s}^t|_{t=1}$ such that
\begin{eqnarray}
H_{s+1}:=H_s\circ \Psi_s=N_{*}+Z_{s+1}+Q_{s+1},
\end{eqnarray}where
\begin{equation}\label{060301}
Z_{s+1}=\sum_{2\leq j\leq s}Z_{(s+1)j},
\end{equation}
\begin{equation}
Q_{s+1}=\sum_{j\geq s+1}Q_{(s+1)j},
\end{equation}
with
\begin{equation}
Z_{(s+1)j}=\sum_{|l|=j}J^{l}\sum_{a,k,k'\atop k=k'}Z_{s+1;akk'}^{(l)}\mathcal{M}_{akk'}
\end{equation}
and
\begin{equation}
Q_{(s+1)j}=\sum_{|l|=j}J^{l}\sum_{a,k,k'}Q_{s+1;akk'}^{(l)}\mathcal{M}_{akk'}.
\end{equation}
Moreover, the following estimates hold:
\begin{equation}\label{063003}
||F_{s}||_{\rho_{s+1}}\leq C_1(\delta,\theta,\gamma)\cdot(C(\delta,\theta,\gamma))^{(s-3)s},
\end{equation}
\begin{equation}\label{062804}
||Z_{(s+1)j}||_{\rho_{s+1}}\leq (C(\delta,\theta,\gamma))^{(s-2)j},
\end{equation}
and
\begin{equation}\label{062807}
||Q_{(s+1)j}||_{\rho_{s+1}}\leq (C(\delta,\theta,\gamma))^{(s-2)j}.
\end{equation}
\end{lemma}
\begin{proof}
The details of the proof will be given in Appendix.
\end{proof}

\textbf{Proof of Theorem \ref{060102}.}
\begin{proof}
In view of Lemma \ref{080301} and Lemma \ref{062808}, we define
\begin{equation}
\Psi:=\Psi_2\circ\cdots\circ\Psi_M.
\end{equation}
Then one has
\begin{equation}
H_{*}\circ\Psi=N_*+Z_{M+1}+Q_{M+1},
\end{equation}
and
$Z_{M+1}$ and $Q_{M+1}$ satisfies the estimates (\ref{062804}) and (\ref{062807}) with $s=M$ respectively.

For fixed $0<\tau\ll1$, we choose
\begin{equation}\label{063008}
M=\frac{1}3|\ln \tau|^{\theta/10}.
\end{equation}
Now we will estimate the norm of the symplectic map $\Psi$ and the remainder term $Q_{M+1}$ respectively.

In view of (\ref{062809}), (i.e. $\delta=\frac{r}{80M}$), (\ref{081002}) and (\ref{081003}), one has
\begin{equation}
C_1(\delta,\theta,\gamma)=\frac{e^3}{\gamma}\cdot e^{C(\theta)\left(\frac{r}{80M}\right)^{-\frac{5}{\theta}}}\leq e^{M^{\frac{6}{\theta}}}
\end{equation}
and
\begin{equation}
C_2(\delta,\theta)=\frac{80M}{r}\cdot\left(\frac{80M}{r}\right)^{C(\theta)\left(\frac{r}{80M}\right)^{-\frac{1}{\theta}}}\leq M^{M^{\frac{2}{\theta}}},
\end{equation}
where letting $M$ large enough depending on $\gamma,r,\theta$.

For $s\geq 3$ in view of (\ref{063003}), one has
\begin{eqnarray}
||F_{s}||_{\rho_{s+1}}
&\leq&\label{063005}e^{M^{\frac{6}{\theta}}}\cdot\left(M^{M^{\frac{2}{\theta}}}\right)^{(s-3)s}\leq e^{M^{\frac{10}{\theta}}} ,
\end{eqnarray}
Based on (\ref{050907}) in Lemma \ref{063004} and (\ref{063005}), one has
\begin{equation}\label{081004}
\sup_{\widetilde D(\tau)}||X_{F_s}||_{r,\infty}\leq C(r,\theta)\cdot e^{M^{10/\theta}}\tau^{s-1},
\end{equation}
where noting that $F_s$ contains $s$ $J's$. Hence, in view of (\ref{063008}) one has
\begin{equation}\label{063009}
||\Psi_s-id||_{r,\infty}\leq C(r,\theta)\tau^{s-\frac{4}3}.
\end{equation}
For $s=2$ and in view of (\ref{081001}), one has
\begin{equation}\label{063010}
||\Psi_2-id||_{r,\infty}\leq \frac12{\tau},
\end{equation}
Based on (\ref{063009}) and (\ref{063010}), we have
 \begin{equation}\label{063011}
||\Psi-id||_{r,\infty}\leq \tau.
\end{equation}

Now we would like to estimate the remainder term $Q_{M+1}$.
In view of (\ref{062809}) and (\ref{062810}), one has
\begin{equation*}
\rho_{M+1}\leq\frac{11r}{80}.
\end{equation*}
Moreover, based on (\ref{062807}) for $s=M$, we have
\begin{equation}\label{063012}
||Q_{(M+1)j}||_{{\frac{11r}{80}}}\leq (C(\delta,\theta,\gamma))^{(M-2)j}.
\end{equation}

Following the proof of (\ref{081004}), one has
\begin{equation*}
\sup_{\widetilde D(\tau)}||X_{Q_{(M+1)j}}||_{r,\infty}\leq C(r,\theta)\tau^{j-\frac32},
\end{equation*}
where noting that $Q_{(M+1)j}$ contains $j$ $J's$. Furthermore, we finish the proof of (\ref{083104}) by noting $M=\frac{1}3|\ln \tau|^{\frac{\theta}{10}}$. Until now, we construct a normal form of order $\frac{1}3|\ln \tau|^{\frac{\theta}{10}}$ around the torus $\mathcal{T}$. Following the proof of Corollary 2.16 in \cite{BG}, it is a standard way to obtain the long time stability of the torus $\mathcal{T}$, i.e. we finish the proof of (\ref{081101}) in Theorem \ref{L10}.

\end{proof}
\section{Appendix}

\begin{lemma}\label{a1}Let $\theta\in(0,1)$ and $k_n,k'_n\in\mathbb{N},|\widetilde{V}_n|\leq2\ \mbox{for}\ \forall \ n\in\mathbb{Z}$.
Assume further
\begin{equation}
\label{041803}\left|\sum_{n\in\mathbb{Z}}(k_n-k_n')(n^2+\widetilde V_n)\right|\leq1
\end{equation}
and
\begin{equation}
\label{041803'}\sum_{n\in\mathbb{Z}}(k_n-k_n')n=0.
\end{equation}
Then one has
\begin{equation}\label{041809'}
\sum_{n\in\mathbb{Z}}|k_n-k_n'||n|^{\theta/2}\leq3\cdot 6^{\theta/2}\sum_{i\geq3}|n_i|^{\theta},
\end{equation}
where $(n_i)_{i\geq1}, |n_1|\geq|n_2|\geq|n_3|\geq\cdots$, denote the system \{$n$: $n$ is repeated $k_n+k'_n$ times\}.
\end{lemma}
\begin{proof}
From the definition of $(n_i)_{i\geq1}$ and (\ref{041803'}), there exist $(\mu_i)_{i\geq1}$ with $\mu_i\in\{\pm1\}$ such that
\begin{equation}
\label{1605141}\sum_{n\in\mathbb{Z}}(k_n-k_n')n^2=\sum_{i\geq1}\mu_in_i^2
\end{equation}
\and
\begin{equation}
\label{1605141'}\sum_{n\in\mathbb{Z}}(k_n-k_n')n=\sum_{i\geq1}\mu_in_i=0.
\end{equation}

In view of (\ref{041803}), (\ref{1605141}) and $|\widetilde V_n|\leq 2$,
one has
\begin{equation*}
\left|\sum_{i\geq1}\mu_in_i^2\right|\leq\left|\sum_{n\in\mathbb{Z}}(k_n-k_n')\widetilde V_n\right|+1\leq2\sum_{n\in\mathbb{Z}}(k_n+k_n')+1,
\end{equation*}which implies
\begin{equation}\label{041804}
\left|n_1^2+\left(\frac{\mu_2}{\mu_1}\right)n_2^2\right|\leq2\sum_{i\geq1}1+\sum_{i\geq3}n_i^2+1\leq \sum_{i\geq 3}(2+n_i^2)+3.
\end{equation}
In the other hand, by (\ref{1605141'}), we obtain
\begin{equation}\label{041805}
\left|n_1+\left(\frac{\mu_2}{\mu_1}\right)n_2\right|\leq \sum_{i\geq 3}|n_i|.
\end{equation}

To prove the inequality (\ref{041809'}), we will distinguish two cases:

\textbf{Case. 1.} $\frac{\mu_2}{\mu_1}=-1$.

\textbf{Case. 1.1.} $n_1=n_2$.

Then it is to show that
\begin{equation*}
\sum_{n\in\mathbb{Z}}|k_n-k_n'||n|^{\theta/2}\leq \sum_{i\geq3}|n_i|^{\theta/2}\leq3\cdot 6^{\theta/2}\sum_{i\geq3}|n_i|^{\theta}.
\end{equation*}

\textbf{Case. 1.2.} $n_1\neq n_2$.

Then one has
\begin{eqnarray}\label{041806}
\nonumber|n_1-n_2|+|n_1+n_2|
&\leq&\nonumber|n_1-n_2|+|n_1^2-n_2^2|\\
& \leq&\nonumber\sum_{i\geq 3}|n_i|+\sum_{i\geq 3}(2+n_i^2)+3\quad \mbox{(in view of (\ref{041804}) and (\ref{041805}))}\\
 &\leq&6\sum_{i\geq3}|n_i|^2.
\end{eqnarray}
Hence
\begin{equation*}
\max\{|n_1|,|n_2|\}\leq \max\{|n_1-n_2|,|n_1+n_2|\}\leq 6\sum_{i\geq3}|n_i|^2,
\end{equation*}
where the last inequality is based on (\ref{041806}).
For $j=1,2,$ one has
\begin{equation*}
|n_j|^{\theta/2}\leq 6^{\theta/2}\left(\sum_{i\geq3}|n_i|^{2}\right)^{\theta/2}\leq  6^{\theta/2}\sum_{i\geq3}|n_i|^{\theta},
\end{equation*}
where the last inequality is based on
the fact that the function $|x|^{\theta/2}$ is a concave function for $0<\theta<1$. Therefore,
\begin{equation}\label{041807}
|n_1|^{\theta/2}+|n_2|^{\theta/2}\leq 2 \cdot 6^{\theta/2}\sum_{i\geq3}|n_i|^{\theta}.
\end{equation}
Now one has
\begin{eqnarray}
\nonumber\sum_{n\in\mathbb{Z}}|k_n-k_n'||n|^{\theta/2}
&\leq&\nonumber\sum_{n\in\mathbb{Z}}(k_n+k_n')|n|^{\theta/2}\\
&=&\nonumber\sum_{i\geq1}|n_i|^{\theta/2}\\
&\leq&\nonumber\left(|n_1|^{\theta/2}+|n_2|^{\theta/2}\right)+\sum_{i\geq3}|n_i|^{\theta}\\
&\leq&\nonumber(2 \cdot 6^{\theta/2}+1)\sum_{i\geq3}|n_i|^{\theta}\ \ \mbox{(in view of (\ref{041807}))}\\
&\leq&\label{2016121101}3 \cdot 6^{\theta/2}\sum_{i\geq3}|n_i|^{\theta}.
\end{eqnarray}
\textbf{Case. 2.} $\frac{\mu_2}{\mu_1}=1$.

In view of (\ref{041804}), one has
\begin{equation*}
n_1^2+n_2^2\leq 5\sum_{i\geq3}|n_i|^{2},
\end{equation*}which implies
\begin{equation*}
|n_j|^{\theta/2}\leq 5^{\theta/2}\left(\sum_{i\geq3}|n_i|^{2}\right)^{\theta/2}\leq  5^{\theta/2}\sum_{i\geq3}|n_i|^{\theta}\ \ \mbox{($j=1,2$)}.
\end{equation*}
Therefore,
\begin{equation}\label{041808}
|n_1|^{\theta/2}+|n_2|^{\theta/2}\leq 2 \cdot 5^{\theta/2}\sum_{i\geq3}|n_i|^{\theta}.
\end{equation}
Following the proof of (\ref{2016121101}), we have
\begin{equation*}
\sum_{n\in\mathbb{Z}}|k_n-k_n'||n|^{\theta/2}\leq
3 \cdot 6^{\theta/2}\sum_{i\geq3}|n_i|^{\theta}.
\end{equation*}

\end{proof}
\begin{lemma}\label{a3}
Assuming $\theta,\delta\in(0,1)$, then we have the following inequality
\begin{equation}\label{041809}
\sum_{a\in\mathbb{N}^{\mathbb{Z}}}e^{-\delta\sum_{n\in\mathbb{Z}}a_{n}|n|^{\theta}}\leq\prod_{n\in\mathbb{Z}}\frac{1}{1-e^{-\delta |n|^{\theta}}},
\end{equation}
where $|0|:=1$.
\end{lemma}

\begin{proof}
$\sum\limits_{a\in\mathbb{N}^{\mathbb{Z}}}e^{-\delta\sum_{n\in\mathbb{Z}}a_{n}|n|^{\theta}}\leq \prod\limits_{n\in\mathbb{Z}}\left(\sum\limits_{a_n\in\mathbb{N}}e^{-\delta a_{n}|n|^{\theta}}\right)
=\prod\limits_{n\in\mathbb{Z}}\frac{1}{1-e^{-\delta |n|^{\theta}}}.$
\end{proof}

\begin{lemma}\label{lem2}
Assuming $\theta,\delta\in(0,1)$, then we have
\begin{equation}\label{0418011}
\sum_{n\geq 1}e^{-\delta n^{\theta}}\leq e^{\frac{(\theta-1)}{\theta}}\cdot\frac 2{\theta}\left(\frac{2(1-\theta)}{\theta}\right)^{\frac{1}{\theta}-1}\delta^{-\frac{1}{\theta}}.
\end{equation}
\end{lemma}

\begin{proof}
Obviously, we have
\begin{eqnarray}
\label{e}\sum_{n\geq 1}e^{-\delta n^{\theta}} \leq \int_{0}^{+\infty}e^{-\delta x^{\theta}}\mathrm{d}x
=\frac{1}{\theta}\int_{0}^{+\infty}e^{-\delta x}x^{\frac{1}{\theta}-1}\mathrm{d}x.
\end{eqnarray}
Let $f(x)=e^{-\frac{1}{2}\delta x}x^{\frac{1}{\theta}-1}$, and it is easy to prove
\begin{equation}
\max_{x\geq 0}f(x)=f\left(\frac{2(1-\theta)}{\delta\theta}\right)=e^{\frac{\theta-1}{\theta}}\left(\frac{2(1-\theta)}{\delta\theta}\right)^{\frac{1}{\theta}-1}.
\end{equation}
Consequently
\begin{eqnarray*}
\nonumber(\ref{e})&\leq& \frac{1}{\theta}\cdot e^{\frac{\theta-1}{\theta}}\cdot\left(\frac{2(1-\theta)}{\delta\theta}\right)^{\frac{1}{\theta}-1}
\int_{0}^{+\infty}e^{-\frac{\delta}{2}x}\mathrm{d}x\\
&\leq&e^{\frac{(\theta-1)}{\theta}}\cdot\frac 2{\theta}\left(\frac{2(1-\theta)}{\theta}\right)^{\frac{1}{\theta}-1}\delta^{-\frac{1}{\theta}},
\end{eqnarray*}
which finishes the proof of (\ref{0418011}).
\end{proof}

\begin{lemma}\label{a5}
Assuming $\theta\in(0,{1}),\delta\in(0,\frac1e)$, then we have
\begin{equation}\label{0418010}
\prod_{n\in\mathbb{Z}}\frac{1}{1-e^{-\delta |n|^{\theta}}}\leq\left(\frac{1}{\delta}\right)^{C(\theta){\delta^{-\frac{1}{\theta}}}},
\end{equation}
where $|0|:=1$ and $C(\theta)$ is a positive constant depending on $\theta$ only.
\end{lemma}

\begin{proof}
We write
\begin{eqnarray}
\nonumber\prod\limits_{n\in\mathbb{Z}}\frac{1}{1-e^{-\delta |n|^{\theta}}}
\nonumber&=&\left(\frac{1}{1-e^{-\delta}}\right)\cdot\prod\limits_{n\geq1}\left(\frac{1}{1-e^{-\delta n^{\theta}}}\right)^2\\
\label{c}&=&\left(\frac{1}{1-e^{-\delta}}\right)\cdot\prod\limits_{1\leq n\leq N_{\theta}}\left(\frac{1}{1-e^{-\delta n^{\theta}}}\right)^2\prod\limits_{n> N_\theta}\left(\frac{1}{1-e^{-\delta n^{\theta}}}\right)^2,\end{eqnarray}
where
\begin{equation}\label{ntheta}
N_{\theta}=\left(\ln(3+2\sqrt{2})\right)^{\frac1\theta}{\delta}^{-\frac1\theta}.
\end{equation}

For $\delta\in(0,\frac1e)$, one has
\begin{equation}\label{042506}\frac{1}{1-e^{-\delta}}\leq \frac{1}{\delta^{2}},
\end{equation} and then
\begin{eqnarray}\nonumber\left(\frac{1}{1-e^{-\delta}}\right)\cdot\prod\limits_{1\leq n\leq N_{\theta}}\left(\frac{1}{1-e^{-\delta n^{\theta}}}\right)^2
\nonumber&\leq& \left(\frac{1}{1-e^{-\delta}}\right)\cdot\prod\limits_{1\leq n\leq N_{\theta}}\left(\frac{1}{1-e^{-\delta }}\right)^2\\
\nonumber&=& \left(\frac{1}{1-e^{-\delta}}\right)^{2N_{\theta}+1}\\
\nonumber&\leq& \left(\frac{1}{\delta}\right)^{4N_{\theta}+2} \ \ \mbox{(in view of (\ref{042506}))}\\
\label{a}&\leq&\left(\frac{1}{\delta}\right)^{6\left(\ln(3+2\sqrt{2})\right)^{\frac1\theta}{\delta}^{-\frac1\theta}}
\ \ \mbox{(in view of (\ref{ntheta}))}.
\end{eqnarray}

When $n>N_{\theta}$, one has
\begin{equation}\label{d}
\delta n^{\theta}>\ln(3+2\sqrt{2}).
\end{equation}
Note that if $x\in(0,3-2\sqrt{2})$, we have $$\ln\left(\frac{1}{1-x}\right)\leq\sqrt{x},$$which implies\begin{equation}\label{042601}\ln\left(\frac1{1-e^{-\delta x}}\right)\leq e^{-\frac{1}{2}\delta x}\ \ \mbox{for $\delta x>\ln(3+2\sqrt{2})$}.\end{equation} Hence
\begin{eqnarray}
\nonumber\prod\limits_{n> N_\theta}\left(\frac{1}{1-e^{-\delta n^{\theta}}}\right)^2
&=& \nonumber e^{\sum_{n> N_\theta}2\ln\left(\frac1{1-e^{-\delta n^{\theta}}}\right)}\\
\nonumber&\leq&e^{\sum_{n> N_\theta}{2e^{-\frac{1}{2}\delta n^{\theta}}}}\quad (\mbox{in view of (\ref{d}) and (\ref{042601})})\\
&\leq&\label{b} e^{C_1({\theta})\delta^{-\frac{1}{\theta}}},
\end{eqnarray}where the last inequality is based on Lemma \ref{lem2} and
\begin{equation*}
C_1(\theta)=2^{\frac1\theta}e^{\frac{(\theta-1)}{\theta}}\cdot\frac 4{\theta}\left(\frac{2(1-\theta)}{\theta}\right)^{\frac{1}{\theta}-1}.
\end{equation*}

Recalling $\delta\in(0,\frac1e)$, then the estimate (\ref{0418010}) follows from (\ref{c}), (\ref{a}) and (\ref{b}), where
\begin{equation*}
C(\theta)=C_1(\theta)+6\left(\ln(3+2\sqrt{2})\right)^{\frac1\theta}.
\end{equation*}
\end{proof}
\begin{lemma}\label{8.6}Assuming $f_p(x)=x^pe^{-\delta x}\ (p=1,2)$, then we have
\begin{equation}\label{042805}
\max_{x\geq0}f_p(x)\leq\frac1{\delta^p}.
\end{equation}
\end{lemma}
\begin{proof}
Since $f_p(x)=x^pe^{-\delta x}$, we have
\begin{equation*}
f_p'(x)=px^{p-1}e^{-\delta x}-\delta x^pe^{-\delta x}=(px^{p-1}-\delta x^p)e^{-\delta x}.
\end{equation*}
Hence we get
\begin{equation*}
f_p'\left(\frac p{\delta}\right)=0
\end{equation*}
and it is easy to see
\begin{equation}\label{042608}
\max_{x\geq0}f_p(x)=f_{p}\left(\frac p{\delta}\right)=\frac{p^p}{\delta^p} e^{-p}\leq \frac{1}{\delta^p}\ \mbox{ ($p=1,2$)}.
\end{equation}
\end{proof}
\begin{lemma}Assuming $\theta,\delta\in(0,1)$ and $a=(a_n)_{n\in\mathbb{Z}}\in\mathbb{N}^{\mathbb{Z}}$, then we have
\begin{equation}\label{042807}
\prod_{n\in\mathbb{Z}}\left(1+a_n^p\right)e^{-2\delta a_n|n|^{\theta}}\leq \left(\frac{1}{\delta}\right)^{3p{\delta}^{-\frac{1}{\theta}}},
\end{equation}
where $p=1,2$ and $|0|:=1$.
\begin{proof}
Firstly, we note
$$\prod_{n\in\mathbb{Z}}\left(1+a_n^p\right)e^{-2\delta a_n|n|^{\theta}}=\prod_{n\in\mathbb{Z}:a_n\geq1}\left(1+a_n^p\right)e^{-2\delta a_n|n|^{\theta}},$$
and we can assume $a_n\geq 1\ \mbox{for}\ \forall\  n\in\mathbb{Z}$ in what follows. Thus one has
\begin{eqnarray}
\nonumber(1+a_n^p)e^{-2\delta a_n|n|^{\theta}}
\nonumber&\leq& 2a_n^pe^{-2\delta a_n|n|^{\theta}}\\
\nonumber&\leq&\frac{1}{|n|^{p\theta}}\cdot(2^{\frac1p}a_n|n|^{\theta})^pe^{-\delta2^{\frac1p} a_n|n|^{\theta}}\\
\label{160516}&\leq& \frac{1}{|n|^{p\theta}}\cdot\frac 1{\delta^p},
\end{eqnarray}
where the last inequality is based on (\ref{042805}). Hence, if $|n|\geq \delta^{-1/\theta}$, one has
\begin{equation}\label{042806}
(1+a_n^p)e^{-2\delta a_n|n|^{\theta}}\leq 1.
\end{equation}
Therefore,
\begin{eqnarray*}
\prod_{n\in\mathbb{Z}}\left(1+a_n^p\right)e^{-2\delta a_n|n|^{\theta}}
&=&\left(\prod_{1\leq|n|<\delta^{-1/\theta}}\left(1+a_n^p\right)e^{-2\delta a_n|n|^{\theta}}\right)\left(\prod_{|n|\geq\delta^{-1/\theta}}\left(1+a_n^p\right)e^{-2\delta a_n|n|^{\theta}}\right)\\
&&\mbox{(noting that $|0|:=1$)}\\
&\leq&\prod_{1\leq|n|<\delta^{-1/\theta}}\left(1+a_n^p\right)e^{-2\delta a_n|n|^{\theta}}\quad \mbox{(in view of (\ref{042806}))}\\
&\leq&\prod_{1\leq|n|<\delta^{-1/\theta}}\frac{1}{|n|^{p\theta}\delta^p}\ \ \mbox{(in view of (\ref{160516}))}\\
&\leq&\left(\frac{1}{\delta}\right)^{2p{\delta}^{-\frac{1}{\theta}}+p}\ \ \ \\
&\leq&\left(\frac{1}{\delta}\right)^{3p{\delta}^{-\frac{1}{\theta}}}.
\end{eqnarray*}
\end{proof}
\end{lemma}
\begin{lemma}Let $a,k,k'\in\mathbb{N}^{\mathbb{Z}},\ \theta\in(0,1)$, and $0<\delta\ll1$ (depending only on $\theta$). Let further
$(n_i)_{i\geq1}, |n_1|\geq|n_2|\geq|n_3|\geq\cdots$, denote the system \{$|n|$: $n$ is repeated $2a_n+k_n+k'_n$ times\}. Then we have
\begin{equation}\label{201606031}
\prod_{|m|\leq|n_1|}(1+l_m^2)e^{-\delta\sum_{|m|\leq|n_3|}l_m|m|^{\theta}}
\leq\left(\frac{1}{\delta}\right)^{C({\theta}){\delta^{-\frac{1}{\theta}}}},
\end{equation}
where
\begin{equation*}
l_n=\#\{j:n=n_j\},
\end{equation*}
and $C(\theta)$ is a positive constant depending only on $\theta$.

\begin{proof}
To prove (\ref{201606031}), we distinguish three cases:

\textbf{Case 1}. $|n_1|=|n_2|=|n_3|$. In this case, we have
\begin{eqnarray}
\nonumber\prod_{|m|\leq|n_1|}(1+l_m^2)e^{-\delta\sum_{|m|\leq|n_3|}l_m|m|^{\theta}}
\nonumber&=&\prod_{|m|\leq|n_1|}(1+l_m^2)e^{-\delta\sum_{i\geq3}|n_i|^{\theta}}\\
\nonumber&\leq&\prod_{|m|\leq|n_1|}(1+l_m^2)e^{-\frac13\delta\sum_{i\geq1}|n_i|^{\theta}}\\
\nonumber&=&\prod_{|m|\leq|n_1|}\left((1+l_m^2)e^{-\frac13\delta l_m|m|^{\theta}}\right)\\
\nonumber&\leq&\left(\frac{6}{\delta}\right)^{6{\left(\frac{6}\delta\right)^{-\frac{1}{\theta}}}}\qquad (\mbox{in view of (\ref{042807})})\\
\label{062101}&\leq&\left(\frac{1}{\delta}\right)^{C_1({\theta}){\delta^{-\frac{1}{\theta}}}},
\end{eqnarray}
where the last inequality relies on $0<\delta\ll 1$ and $C_1(\theta)$ is a positive constant depending on $\theta$ only.

\textbf{Case 2}. $|n_1|>|n_2|=|n_3|$. In this case, $l_{n_1}=1$. Hence, we have
\begin{eqnarray}
\nonumber\prod_{|m|\leq|n_1|}(1+l_m^2)e^{-\delta\sum_{|m|\leq|n_3|}l_m|m|^{\theta}}
\nonumber&=&2\cdot\prod_{|m|\leq|n_2|}(1+l_m^2)e^{-\delta\sum_{i\geq3}|n_i|^{\theta}}\\
\nonumber&\leq&2\cdot\prod_{|m|\leq|n_2|}(1+l_m^2)e^{-\frac{1}{2}\delta\sum_{i\geq2}|n_i|^{\theta}}\\
\nonumber&=&2\cdot\prod_{|m|\leq|n_2|}\left((1+l_m^2)e^{-\frac12\delta l_m|m|^{\theta}}\right)\\
\label{062102}&\leq&\left(\frac{1}{\delta}\right)^{C_2({\theta}){\delta^{-\frac{1}{\theta}}}},
\end{eqnarray}
where the last inequality follows form the proof of (\ref{062101}) and $C_2(\theta)$ is a positive constant depending on $\theta$ only.

\textbf{Case 3}. $|n_1|\geq|n_2|>|n_3|$. In this case, $l_{m}\leq2$ for $m\in\{n_1,n_2\}$. Thus, we have
\begin{eqnarray}
\nonumber\prod_{|m|\leq|n_1|}(1+l_m^2)e^{-\delta\sum_{|m|\leq|n_3|}l_m|m|^{\theta}}
\nonumber&=&\prod_{m\in\{n_1,n_2\}}(1+l_m^2)\prod_{|m|\leq|n_3|}\left((1+l_m^2)e^{-\delta l_m|m|^\theta}\right)\\
\nonumber&\leq&5\cdot\prod_{|m|\leq|n_3|}\left((1+l_m^2)e^{-\delta l_m|m|^\theta}\right)\\
\label{062103}&\leq&\left(\frac{1}{\delta}\right)^{C_3({\theta}){\delta^{-\frac{1}{\theta}}}},
\end{eqnarray}
where the last inequality follows form the proof of (\ref{062101}) and $C_3(\theta)$ is a positive constant depending on $\theta$ only.

In view of (\ref{062101})-(\ref{062103}), we finished the proof of (\ref{201606031}).
\end{proof}
\end{lemma}

\textbf{The proof of Lemma \ref{051301}}
\begin{proof}
Firstly, we will prove the inequality (\ref{N6}).
Write $\mathcal{M}_{akk'}$ in the form of
\begin{equation*}
\mathcal{M}_{akk'}=\mathcal{M}_{abll'}=\prod_nI_n(0)^{a_n}I_n^{b_n}q_n^{l_n}{\bar q_n}^{l_n'}
\end{equation*}
where
\begin{equation*}
b_n=k_n\wedge k_n',\quad l_n=k_n-b_n,\quad l_n'=k_n'-b_n'
\end{equation*}
and
$l_nl_n'=0$ for all $n$.

Express the term
\begin{equation*}
\prod_nI_n^{b_n}=\prod_n(I_n(0)+J_n)^{b_n}
\end{equation*}by the monomials of the form
\begin{equation*}
\prod_nI_n(0)^{b_n},
\end{equation*}
\begin{equation*}
\sum_{m,b_m\geq 1}\left(I_m(0)^{b_m-1}J_m\right)\left(\sum_{n\neq m}\prod_nI_n(0)^{b_n}\right),
\end{equation*}
\begin{equation*}
\sum_{m,b_m\geq2\atop
r\leq b_m-2}\left(\sum_{n< m}\prod_nI_n(0)^{b_n}\right)\left(b_m(b_m-1)I_m(0)^{r}J_m^2I_m^{b_m-r-2}\right)\left(\sum_{n> m}\prod_nI_n^{b_n}\right),
\end{equation*}
and
\begin{eqnarray*}
&&\sum_{m_1< m_2,b_{m_1},b_{m_2}\geq 1\atop
r\leq b_{m_2}-1}\left(\sum_{n< m_1}\prod_nI_n(0)^{b_n}\right)\left(b_{m_1}I_{m_1}(0)^{b_{m_1}-1}J_{m_1}\right)
\\
&&\nonumber\times\left(\sum_{m_1<n< m_2}\prod_nI_n^{b_n}\right)\left(b_{m_1}I_{m_1}(0)^{r}J_{m_1}I_{m_1}^{b_{m_1}-r-1}\right)\left(\sum_{n> m_2}\prod_nI_n^{b_n}\right).
\end{eqnarray*}
Now we will estimate the bounds for the coefficients respectively. Consider the term
$\mathcal{M}_{akk'}=\prod_nI_n(0)^{a_n}q_n^{k_n}\bar q_n^{k_n'}$ with fixed $a,k,k'$ satisfying $k_nk_n'=0$ for all $n$. It is easy to see that $\mathcal{M}_{akk'}$ comes from some parts of the terms $\mathcal{M}_{\alpha\kappa\kappa'}$ with no assumption for $\kappa$ and $\kappa'$. For any given $n$ one has
\begin{equation*}
I_n(0)^{a_n}q_n^{k_n}\bar q_n^{k_n'}=\sum_{\beta_n=k_n\wedge k_n'}I_n(0)^{\alpha_n+\beta_n}q_n^{\kappa_n-\beta_n}\bar q_n^{\kappa_n'-\beta_n}.
\end{equation*}
Hence,
\begin{equation}\label{N2}
\alpha_n+\beta_n=a_n,
\end{equation}
and
\begin{equation}\label{N3}
\kappa_n-\beta_n=k_n,\qquad \kappa_n'-\beta_n=k_n'.
\end{equation}
Therefore, if $0\leq\alpha_n\leq a_n$ is chosen, so $\beta_n,k_n,k_n'$ are determined.
On the other hand,
\begin{eqnarray}
\nonumber|B_{\alpha\kappa\kappa'}|
&\leq&\nonumber ||R||_{\rho}e^{\rho\left(\sum_{n}(2\alpha_n+\kappa_n+\kappa_n')|n|^{\theta}-2(n_1^*)^{\theta}\right)}\qquad\qquad\qquad\qquad\ \\&&\nonumber{(\mbox{in view of (\ref{042602})})}\\
&=&\nonumber||R||_{\rho}e^{\rho\left(\sum_{n}(2\alpha_n+(k_n+a_n-\alpha_n)+(k_n'+a_n-\alpha_n))|n|^{\theta}-2(n_1^*)^{\theta}\right)}\quad\\
&&\nonumber{(\mbox{in view of (\ref{N2}) and (\ref{N3})})}\\
&=&\nonumber||R||_{\rho}e^{\rho\left(\sum_{n}(2a_n+k_n+k_n')|n|^{\theta}-2(n_1^*)^{\theta}\right)}.
\end{eqnarray}
Hence,
\begin{equation}\label{N4}
|B_{akk'}|\leq||R||_{\rho}\prod_n(1+a_n)e^{\rho\left(\sum_{n}(2a_n+k_n+k_n')|n|^{\theta}-2(n_1^*)^{\theta}\right)}.
\end{equation}
Similarly,
\begin{eqnarray*}
|B_{akk'}^{(m)}|&\leq& ||R||_{\rho}\left(\prod_{n\neq m}(1+a_n)\right)(1+a_m)^2e^{\rho\left(\sum_{n}(2a_n+k_n+k_n')|n|^{\theta}+2|m|^{\theta}-2(n_1^*)^{\theta}\right)},\\
|B_{akk'}^{(m,m)}|&\leq& ||R||_{\rho}\left(\prod_{n\neq m}(1+a_n)\right)(1+a_m)^3e^{\rho\left(\sum_{n}(2a_n+k_n+k_n')|n|^{\theta}+4|m|^{\theta}-2(n_1^*)^{\theta}\right)},\\
|B_{akk'}^{(m_1,m_2)}|&\leq& ||R||_{\rho}\left(\prod_{n<m_1}(1+a_n)\right)(1+a_{m_1})^2\left(\prod_{m_1<n<m_2 }(1+a_n)\right)\\&&\times (1+a_{m_2})^2e^{\rho\left(\sum_{n}(2a_n+k_n+k_n')|n|^{\theta}+2|m_1|^{\theta}+2|m_2|^{\theta}-2(n_1^*)^{\theta}\right)}.
\end{eqnarray*}

In view of (\ref{051302}) and (\ref{N4}), we have
\begin{eqnarray}\label{N5}
||R_0||_{\rho+\delta}^{+}
\leq||R||_{\rho}\prod_n(1+a_n)e^{-\delta\left(\sum_{n}(2a_n+k_n+k_n')|n|^{\theta}-2(n_1^*)^{\theta}\right)}.
\end{eqnarray}
Now we will show that
\begin{eqnarray}
\label{051303}\prod_n(1+a_n)e^{-\delta\left(\sum_{n}(2a_n+k_n+k_n')|n|^{\theta}-2(n_1^*)^{\theta}\right)}&\leq& \left(\frac{1}{\delta}\right)^{ C(\theta)\delta^{-\frac{1}{\theta}}},
\end{eqnarray}
where $C(\theta)$ is a positive constant depending only on $\theta$.

\textbf{Case 1.} $n_1^*=n_3^*.$
Then one has
\begin{eqnarray*}
(\ref{051303})
&=&\nonumber \prod_n(1+a_n)e^{-{\delta}\sum_{i\geq3}|n_i|^{\theta}}\\
&\leq&\nonumber\prod_n(1+a_n)e^{-\frac{\delta}{3}\sum_{i\geq1}|n_i|^{\theta}}\\
&=&\nonumber\prod_n(1+a_n)e^{-\frac\delta3\sum_{n}(2a_n+k_n+k_n')|n|^{\theta}}\\
&\leq&\nonumber\prod_n\left((1+a_n)e^{-\frac{2\delta}{3}a_n|n|^\theta}\right)\\
&\leq&\left(\frac{1}{\delta}\right)^{C(\theta){\delta}^{-\frac{1}{\theta}}}\ \ \mbox{(in view of (\ref{042807}))}.
\end{eqnarray*}

\textbf{Case 2.} $n_1^*>n_2^*=n_3^*.$ In this case, $a_{n}=1$ for $n=n_1$.
Then we have
\begin{eqnarray*}
(\ref{051303})
&=&(1+a_{n_1})\left(\prod_{|n|\leq n_2^*}(1+a_n)e^{-(2-2^\theta)\delta\sum_{i\geq3}(n_i^*)^{\theta}}\right)
\\
&\leq&2\cdot\prod_{|n|\leq n_2^*}(1+a_n)e^{-\frac12(2-2^\theta)\delta\sum_{i\geq2}(n_i^*)^{\theta}}\\
&=&\nonumber2\cdot\prod_{|n|\leq n_2^*}(1+a_n)e^{-\frac12(2-2^\theta)\delta\sum_{|n|\leq n_2^*}(2a_n+k_n+k_n')|n|^{\theta}}\\
&\leq&2\cdot\prod_{|n|\leq n_2^*}\left( (1+a_n)e^{-(2-2^\theta)\delta a_n|n|^\theta}\right)\\
&\leq&\left(\frac{1}{\delta}\right)^{C(\theta){\delta}^{-\frac{1}{\theta}}}\ \ \mbox{(in view of (\ref{042807}))}.
\end{eqnarray*}

\textbf{Case 3.} $n_2^*>n_3^*.$ In this case, $a_{n}\leq2$ for $n\in\{n_1, n_2\}$.
 Hence
\begin{eqnarray*}
(\ref{051303})
&\leq&\left(\prod_{n\in \{n_1,n_2\}}(1+a_n)\right)\left(\prod_{|n|\leq n_3^*}(1+a_n)e^{-\delta\sum_{i\geq3}(n_i^*)^{\theta}}\right)
\\
%&&\mbox{(in view of (\ref{}))}\\
&\leq&\nonumber2^2\cdot\prod_{|n|\leq n_3^*}(1+a_n)e^{-\delta\sum_{|n|\leq n_3^*}(2a_n+k_n+k_n')|n|^{\theta}}\\
&\leq&2^2\cdot\prod_{|n|\leq n_3^*}\left( (1+a_n)e^{-2\delta a_n|n|^\theta}\right)\\
&\leq&\left(\frac{1}{\delta}\right)^{C(\theta){\delta}^{-\frac{1}{\theta}}}\ \ \mbox{(in view of (\ref{042807}))}.
\end{eqnarray*}

We finished the proof of ({\ref{051303}}).

Similarly, one has
\begin{eqnarray*}
&&||R_1||_{\rho+\delta}^{+},||R_2||_{\rho+\delta}^{+}\leq\left(\frac{1}{\delta}\right)^{C(\theta)\delta^{-\frac{1}{\theta}}},
\end{eqnarray*}
and hence
\begin{equation*}
||R||_{\rho+\delta}^{+}\leq\left(\frac{1}{\delta}\right)^{ C(\theta)\delta^{-\frac{1}{\theta}}}||R||_{\rho}.
\end{equation*}

On the other hand, the coefficient of $\mathcal{M}_{abll'}$ increases by at most a factor $(\sum_{n}a_n+b_n)^2$, then
\begin{eqnarray}
\nonumber||R||_{\rho+\delta}
&\leq&\nonumber||R||_{\rho}^{+}\left(\sum_{n}a_n+b_n\right)^2e^{-\delta(\sum_{n}(2a_n+k_n+k_n')|n|^{\theta}-2(n_1^*)^{\theta})}\\
\label{4.24}&\leq&||R||_{\rho}^{+}\left(2\sum_{i\geq 3}(n_i^*)^{\theta}\right)^2 e^{-\delta(2-2^\theta)\sum_{i\geq3}(n_i^*)^{\theta}}\quad (\mbox{in view of (\ref{042605})})\\
&\leq&\nonumber\frac{16}{(2-2^{\theta})^2\delta^2}||R||_{\rho}^{+},
\end{eqnarray}
where the last inequality is based on Lemma \ref{8.6} with $p=2$.
\end{proof}

\textbf{The proof of Lemma \ref{062808}}
\begin{proof}
\textbf{Step 1. The derivative of the homological equation}

Let
\begin{equation*}
F_{s}=\sum_{|l|=s}J^{l}\sum_{a,k,k'}F_{s;akk'}^{(l)}\mathcal{M}_{akk'},
\end{equation*}
and let $\Psi_{s}=X_{F_{s}}^t|_{t=1}$ be the time-1 map of the Hamiltonian vector field $X_{F_{s}}$.

Using Taylor's formula,
\begin{eqnarray}
\nonumber H_{{s}+1}&:=&H_{s}\circ X_{F_{s}}^t|_{t=1}\\
&=&\nonumber(N_*+Z_{s}+Q_{s})\circ X_{F_{s}}^t|_{t=1}\\
&=&\nonumber N_*+\{N_*,F_{s}\}+\sum_{n\geq2}\frac{1}{n!}\underbrace{\{\cdots\{N_*,{F_{s}}\},{F_{s}},\cdots,{F_{s}}\}}_{n-\mbox{fold}}\\
&&\nonumber+Q_{ss}+\sum_{n\geq1}\frac{1}{n!}\underbrace{\{\cdots\{Q_{ss},{F_{s}}\},{F_{s}},\cdots,{F_{s}}\}}_{n-\mbox{fold}}\\
&&\nonumber+\sum_{n\geq0}\frac{1}{n!}\underbrace{\{\cdots\{Z_{s}+Q_{s}-Q_{ss},{F_{s}}\},{F_{s}},\cdots,{F_{s}}\}}_{n-\mbox{fold}}.
\end{eqnarray}
Now we obtain the homological equation
\begin{equation}\label{060201}
\{N_*,F_{s}\}+Q_{ss}={Z_{ss}},
\end{equation}
where
\begin{equation*}
{Z_{ss}}=\sum_{|l|=s}J^{l}\sum_{a,k,k'\atop k=k'}Q_{s;akk'}^{(l)}\mathcal{M}_{akk'}.
\end{equation*}
If the homological equation (\ref{060201}) is solvable, then we define
\begin{equation}\label{060302}
Z_{s+1}=Z_{s}+{Z_{ss}},
\end{equation}
and
\begin{eqnarray}
\nonumber Q_{s+1}
&=&\label{060701}\sum_{n\geq2}\frac{1}{n!}\underbrace{\{\cdots\{N_*,{F_{s}}\},{F_{s}},\cdots,{F_{s}}\}}_{n-\mbox{fold}}\\
&&\label{060702}+\sum_{n\geq1}\frac{1}{n!}\underbrace{\{\cdots\{Q_{ss},{F_{s}}\},{F_{s}},\cdots,{F_{s}}\}}_{n-\mbox{fold}}\\
&&\label{060703}+\sum_{n\geq0}\frac{1}{n!}\underbrace{\{\cdots\{Z_{s}+Q_{s}-Q_{ss},{F_{s}}\},{F_{s}},\cdots,{F_{s}}\}}_{n-\mbox{fold}},
\end{eqnarray}
and one has
\begin{equation*}
H_{s+1}=N_*+Z_{s+1}+Q_{s+1}.
\end{equation*}

\textbf{Step 2. The solution of the homological equation (\ref{060201}).}
It is easy to show that the solution of the homological equation is given by
\begin{equation*}
F_{s;akk'}^{(l)}=\frac{Q_{s;akk'}^{(l)}}{\sum_n(k_n-k'_n)(n^2+\omega_n)}.
\end{equation*}
In view of the fact that $\omega$ is Diophantine and following the proof of Lemma \ref{S6}, one has
\begin{eqnarray}
||F_{s}||_{\rho_{s}+\delta}&\leq&\nonumber C_1(\delta,\theta,\gamma)\cdot||Q_{ss}||_{\rho_{s}}\\
&\leq& \label{062805} C_1(\delta,\theta,\gamma)(C(\delta,\theta,\gamma))^{(s-2)s},
\end{eqnarray}
where the last inequality is based on (\ref{062801}) for $j=s$.

\textbf{Step 3. Estimate the remainder terms $Z_{s+1}$ and $Q_{s+1}$.}

Following the notation as (\ref{060301}), rewrite $Z_{s+1}$ as
\begin{equation*}
Z_{s+1}=\sum_{3\leq j\leq s}Z_{(s+1)j},
\end{equation*}
where
\begin{equation*}
Z_{(s+1)j}=\sum_{|l|=j}J^{l}\sum_{a,k,k'\atop k=k'}Z_{s+1;akk'}^{(l)}\mathcal{M}_{akk'}.
\end{equation*}
In view of (\ref{060301}) and (\ref{060302}), one has
\begin{equation*}
Z_{(s+1)j}=Z_{sj},\qquad 3\leq j\leq s.
\end{equation*}
For $3\leq j\leq s-1$, one has
\begin{eqnarray}
\nonumber||Z_{(s+1)j}||_{\rho_{s+1}}
&=&\nonumber
||Z_{sj}||_{\rho_{s+1}}\\
&\leq&\nonumber
||Z_{sj}||_{\rho_{s}}\\
&\leq&\nonumber(C(\delta,\theta,\gamma))^{(s-2)j}\\
&\leq&\label{062802}(C(\delta,\theta,\gamma))^{(s-1)j}.
\end{eqnarray}
In view of (\ref{062802}) and (\ref{062803}), we finish the proof of
When $j=s$, one has
\begin{eqnarray}
\nonumber||Z_{(s+1)s}||_{\rho_{s+1}}
&=&\nonumber
||Q_{ss}||_{\rho_{s+1}}\\
&\leq&\nonumber
||Q_{ss}||_{\rho_{s}}\\
&\leq&\nonumber(C(\delta,\theta,\gamma))^{(s-2)j}\\
&\leq&\label{062803}(C(\delta,\theta,\gamma))^{(s-1)j}.
\end{eqnarray}
In view of (\ref{062802}) and (\ref{062803}), we finish the proof of (\ref{062804}).

Now we would like to estimate the norm of (\ref{060701})-(\ref{060703}).
Without loss of generality, we only consider the following term
\begin{equation*}
\frac{1}{n!}\underbrace{\{\cdots\{Q_{ss},{F_{s}}\},{F_{s}},\cdots,{F_{s}}\}}_{n-\mbox{fold}},
\end{equation*}
which is in (\ref{060702}) and contains at least $s+n(s-1)$ $J's$. Then

\begin{eqnarray}
\nonumber\left|\left|\frac{1}{n!}\underbrace{\{\cdots\{Q_{ss},{F_{s}}\},{F_{s}},\cdots,{F_{s}}\}}_{n-\mbox{fold}}\right|\right|_{\rho_{s+1}}
&=&\nonumber\left|\left|\frac{1}{n!}\underbrace{\{\cdots\{Q_{ss},{F_{s}}\},{F_{s}},\cdots,{F_{s}}\}}_{n-\mbox{fold}}\right|\right|_{\rho_{s}+2\delta}\\
&\leq&\nonumber\frac{1}{n!}\left(C_2(\delta,\theta)
||F_{s}||_{\rho_{s}+\delta}\right)^n\left(\frac{n}{\delta}\right)^n||Q_{ss}||_{\rho_{s}+\delta}\\
&&\nonumber(\mbox{following the proof of (\ref{3.3})})\\
&\leq&\nonumber\frac{1}{n!}\left(C_2(\delta,\theta)\cdot C_1(\delta,\theta,\gamma)\right)^n\left(\frac{n}{\delta}\right)^n||Q_{ss}||_{\rho_{s}}^{n+1}\\
&&(\mbox{in view of the first inequality in (\ref{062805})})\nonumber\\
&\leq&\nonumber\left(\frac{e}{\delta}\cdot C_2(\delta,\theta)\cdot C_1(\delta,\theta,\gamma)\right)^n||Q_{ss}||_{\rho_{s}}^{n+1}\qquad \\
&&(\mbox{in view of $n^n/n!\leq e^n$})\\
&\leq&\nonumber\left( C(\delta,\theta,\gamma)\right)^n \left((C(\delta,\theta,\gamma))^{(s-2)s}\right)^{n+1}\\
&&(\mbox{in view of (\ref{062806}) and (\ref{062801}) for $j=s$})\nonumber\\
&=&\nonumber(C(\delta,\theta,\gamma))^{(s-2)s(n+1)+n}\\
&\leq&(C(\delta,\theta,\gamma))^{(s-1)(s+n(s-1))},
\end{eqnarray}
where the last inequality is based on
\begin{equation*}
(s-2)s(n+1)+n\leq (s-1)(s+n(s-1))
\end{equation*}for $s\geq 2$ and any $n\geq0$. Hence, we finish the proof of (\ref{062807}).
\end{proof}

\bibliographystyle{amsplain}

\end{document}